\documentclass[11pt]{article}
\usepackage[utf8]{inputenc}
\usepackage[dvipsnames]{xcolor}
\usepackage{graphicx}
\usepackage{amsmath}
\usepackage{amsfonts}
\usepackage{amssymb}
\usepackage{amsthm}
\usepackage{thmtools}
\usepackage[margin=25mm]{geometry}

\usepackage{caption}
\usepackage{calc}
\usepackage[shortlabels]{enumitem}
\setlist[itemize]{topsep=0ex,itemsep=0.5ex,parsep=0ex}
\setlist[enumerate]{topsep=0ex,itemsep=0.5ex,parsep=0ex}
\usepackage[longnamesfirst,numbers,sort&compress]{natbib}
\makeatletter
\def\NAT@spacechar{~}
\makeatother
\usepackage[unicode=true]{hyperref}
\hypersetup{
colorlinks,
linkcolor={black},
citecolor={black},
urlcolor={blue!60!black}
}
\usepackage[noabbrev,capitalise,nameinlink]{cleveref}
\crefname{lemma}{Lemma}{Lemmas}
\crefname{theorem}{Theorem}{Theorems}
\crefname{corollary}{Corollary}{Corollaries}
\crefname{proposition}{Proposition}{Propositions}
\crefname{figure}{Figure}{Figures}
\theoremstyle{plain}
\newtheorem{theorem}{Theorem}
\newtheorem{proposition}[theorem]{Proposition}
\newtheorem{lemma}[theorem]{Lemma}
\newtheorem{corollary}[theorem]{Corollary}

\newtheorem{observation}[theorem]{Observation}

\newenvironment{proofofclaim}[1][Proof.]{%
    \begin{proof}[{#1}]%
        }{%
    \end{proof}}
\theoremstyle{definition}

\DeclareMathOperator{\dist}{dist}
\DeclareMathOperator{\diam}{diam}
\DeclareMathOperator{\wdiam}{wdiam}
\DeclareMathOperator{\ANdim}{ANdim}

\DeclareMathOperator{\Int}{Int}

\DeclareMathOperator{\Part}{Part}
\DeclareMathOperator{\Parts}{Parts}
\DeclareMathOperator{\UNW}{unweighted}
\DeclareMathOperator{\rad}{rad}
\DeclareMathOperator{\asdim}{asdim}
\newcommand{\scr}[1]{\mathcal{#1}}
\newcommand{\ds}[1]{\mathbb{#1}}
\newcommand{\eps}{\epsilon}
\newcommand{\UWT}[2]{#1\langle #2\rangle}
\newcommand{\WT}[2]{#1\langle \widehat{#2}\rangle}
\newcommand{\defn}[1]{\textcolor{Maroon}{\emph{#1}}}
\newcommand{\mathdefn}[1]{\textcolor{Maroon}{#1}}
\renewcommand{\thefootnote}{\arabic{footnote}}
\renewcommand{\thefootnote}{\fnsymbol{footnote}}
\renewcommand{\geq}{\geqslant}
\renewcommand{\leq}{\leqslant}

\begin{document}

\title{\bf Proper Minor-Closed Classes of Graphs\\ have Assouad--Nagata Dimension 2}
\author{Marc Distel\,\footnotemark[2]}
\date{\today}

\footnotetext[2]{School of Mathematics, Monash University, Melbourne, Australia (\texttt{marc.distel@monash.edu}). Research supported by Australian Government Research Training Program Scholarship.}

\maketitle

\abstract{Asymptotic dimension and Assouad--Nagata dimension are measures of the large-scale shape of a class of graphs. Bonamy, Bousquet, Esperet, Groenland, Liu, Pirot, and Scott~[\emph{J.~Eur.~Math.~Society}] showed that any proper minor-closed class has asymptotic dimension 2, dropping to 1 only if the treewidth is bounded. We improve this result by showing it also holds for the stricter Assouad--Nagata dimension. We also characterise when subdivision-closed classes of graphs have bounded Assouad--Nagata dimension.}

\setlength{\parindent}{0pt}
\setlength{\parskip}{10pt}

\section{Introduction}
\renewcommand{\thefootnote}{\arabic{footnote}}

    Asymptotic dimension, denoted by $\asdim$, is a measure of the large-scale shape of a metric space. First introduced by \citet*{Gromov1993} for the geometric study of groups, it has since been studied for metrics induced by graphs. \citet*{Bonamy} established the following breakthrough result improving or generalising previous work in \citep{Bell2008,Fujiwara2021,Ostrovskii2015}.
    
    \begin{theorem}[\citep{Bonamy}]
        \label{asdimMCC}
        For every proper minor-closed class of graphs $\scr{G}$,
        \begin{enumerate}[label=(\alph*)]
            \item $\asdim(\scr{G})\leq 2$; and
            \item $\asdim(\scr{G})\leq 1$ if and only if $\scr{G}$ has bounded treewidth.
        \end{enumerate}
    \end{theorem}

    See \cref{secasdim} for a formal definition of asymptotic dimension.

    Asymptotic dimension has also been studied for various non-minor-closed classes in \citep{Bonamy,Papasoglu2023,Dvorak2022}.

    Another parameter strongly related to asymptotic dimension is Assouad--Nagata dimension, denoted by $\ANdim$. Assouad--Nagata dimension can be viewed as a linear strengthening of asymptotic dimension, although Assouad--Nagata dimension was originally introduced by \citet*{Nagata1958} prior to the work of \citet*{Gromov1993} and the definition of asymptotic dimension. The original definition of Assouad--Nagata dimension is instead analogous to Lebesgue dimension; it is due to an equivalent definition introduced by \citet*{Assouad1982} that we can view Assouad--Nagata dimension as analogous to asymptotic dimension instead. See \cref{secasdim} for a formal definition.

    Asymptotic and Assouad--Nagata dimension have applications in geometry group theory \citep{MR2383529,MR2596064,Gromov1993} and for embeddings into Banach spaces \citep{MR2367208,Lang2005}. Furthermore, Assouad--Nagata dimension has recently been tied to finding good approximations for the travelling salesman problem \citep{Erschler}. It is immediate from the definition of Assouad--Nagata dimension that it is at least the asymptotic dimension, however no upper bound on the Assouad--Nagata dimension in terms of the asymptotic dimension is possible; indeed, the class of graphs that can be embedded in the plane with at most one crossing per edge has asymptotic dimension 2 but infinite Assouad--Nagata dimension \citep{Bonamy}. 
    
    \citet*{Bonamy} also showed the following.
    \begin{lemma}
        \label{ANdimGenus}
        For every integer $g\geq 0$, the class of weighted graphs with Euler genus at most $g$ has Assouad--Nagata dimension 2.
    \end{lemma}

    The following question arises naturally from the work of \citet*{Bonamy}: can \cref{asdimMCC} be strengthened by replacing asymptotic dimension with Assouad--Nagata dimension. While it was already known that the class of $K_t$-minor-free graphs has Assouad--Nagata dimension at most $4^t-1$ \citep{Ostrovskii2015}, we show that the bound in \cref{asdimMCC} can be matched exactly in the stronger setting of Assouad--Nagata dimension.

    \begin{theorem}
        \label{ANdimMCC}
        For every proper minor-closed class of graphs $\scr{G}$,
        \begin{enumerate}[label=(\alph*)]
            \item $\ANdim(\scr{G})\leq 2$; and
            \item $\ANdim(\scr{G})\leq 1$ if and only if $\scr{G}$ has bounded treewidth.
        \end{enumerate}
    \end{theorem} 

    We conclude this paper with an exact characterisation of when subdivision-closed classes of graphs have bounded Assouad--Nagata dimension.

    \begin{theorem}
        \label{ANdimSubdivide}
        Let $\scr{G}$ be a subdivision-closed graph class. Then $\scr{G}$ has bounded Assouad--Nagata dimension if and only if there is a graph $H$ such that $\scr{G}$ is $H$-minor-free.
    \end{theorem}

    \cref{ANdimSubdivide} immediately implies the previously mentioned result of \citet*{Bonamy} that the class of graphs that can be embedded in the plane with at most one crossing per edge has infinite Assouad--Nagata dimension, since this class is subdivision-closed and admits every graph $H$ as a minor.

    This paper uses many similar techniques to \citet*{Bonamy}; in particular \cref{ANdimSCLemma} loosely follows the method of Lemma~3.2 in \citet*{Bonamy}. Our novel contribution to this method is to use the ``completion" of the graph, where edges in the torso are weighted in order to preserve the metric exactly, as opposed to the gadgets used in \citet*{Bonamy} which only approximated the metric. Our other major novel contribution in this paper regards how almost-embeddedable graphs are handled; the technique used in \citet*{Bonamy} cannot be applied, and while the vortices and the graph embedded on the surface can be handled individually with relevant ease, handling their union is challenging as the union would normally destroy any control functions. We solve this issue by showing that neighbourhood around where the vortices meet the embedded graph is ``well-behaved", and show that this is sufficient to allow us to extend our control functions for the vortices and the embedded graph to a control function for the union.

\subsection{Basic definitions}

    Let $\ds{R}^+$ denote the set of all strictly positive real numbers, and let $\ds{N}:=\{1,2,\dots\}$.

    A \defn{graph} $G$ consists of a \defn{vertex-set} $V(G)$ and an \defn{edge-set} $E(G)\subseteq \binom{V(G)}{2}$. $G$ is \defn{finite} if $V(G)$ is finite, otherwise $G$ is \defn{infinite}. Unless otherwise stated, all graphs are finite, hence all further definitions are given assuming graphs are finite, although we remark that most of the definitions extend to infinite graphs.

    A \defn{weighted graph} is a graph $G$ equipped with a \defn{weighting} $w:E(G)\rightarrow \ds{R}^+ \cup \{\infty\}$. If a graph $G$ has not been equipped with a weighting, it is \defn{unweighted}. Sometimes, we consider unweighted graphs to be a special case of weighted graphs by equipping them with the ``default" weighting in which every edge gets weight $1$.

    Given an graph $G$ and a function $w:E(G)\rightarrow \ds{R}^+\cup \{\infty\}$, to \defn{imbue} $G$ with $w$ is to create a new weighted graph with the same vertex and edge sets as $G$ but with weighting $w$. By contrast, we can remove the weighting of a weighted graph $G$ to recover the underlying unweighted graph, which we denote \defn{$\UNW(G)$}.

    A graph $H$ is a \defn{subgraph} of a graph $G$, denoted $\mathdefn{H\subseteq G}$, if $V(H)\subseteq V(G)$ and $E(H)\subseteq E(G)$; if $G$ is a weighted graph, we further require that $H$ inherits the weighting of $G$. Two subgraphs $H_1,H_2$ of $G$ are \defn{adjacent} in $G$ if there exists $u\in V(H_1)$ and $v\in V(H_2)$ such that $uv\in E(G)$. The subgraph of $G$ \defn{induced} by $S\subseteq V(G)$, denoted \defn{$G[S]$}, is the subgraph $H$ of $G$ with $V(H):=S$ and $E(H):=\binom{S}{2}\cap E(G)$. A subgraph $H$ of $G$ is an \defn{induced subgraph} of $G$ if $H=G[V(H)]$. For $S\subseteq V(G)$, let $\mathdefn{G-S}:=G[V(G)\setminus S]$.

    Given an edge $uv$ of a graph $G$, for an integer $n\geq 0$, to \defn{subdivide} $uv$ $n$ times is to replace $uv$ with a path with $n$ internal vertices. To \defn{contract} $uv$ is to delete $u$ and $v$ and add a new vertex $x$ adjacent to every vertex that was adjacent to at least one of $u$ or $v$. A graph $H$ is a \defn{minor} of a $G$, denoted \defn{$H\leq G$}, if a graph isomorphic to $H$ can be obtained from a subgraph of $G$ by performing any number of contractions; otherwise, $G$ is \defn{$H$-minor-free}.

    A \defn{class} of graphs is a collection of graphs that is closed under isomorphism. A graph class $\scr{G}$ is \defn{hereditary} if every induced subgraph of a graph in $\scr{G}$ is also in $\scr{G}$. $\scr{G}$ \defn{admits} an $H$-minor if $H$ is a minor of some graph in $\scr{G}$; otherwise $\scr{G}$ is \defn{$H$-minor-free}. $\scr{G}$ is \defn{subdivision-closed} if for every $G\in\scr{G}$, every subdivision of $G$ is in $\scr{G}$. $\scr{G}$ is \defn{minor-closed} if for every $G\in\scr{G}$, every minor of $G$ is in $\scr{G}$. A minor-closed class $\scr{G}$ is \defn{proper} if it is not the class of all graphs; in this case, $\scr{G}$ must be $H$-minor-free for some $H$.

    A \defn{$(d,m)$-grid} is a graph that is isomorphic to the graph with vertex set $\{1,\dots,m\}^d$ and edges between two vertices only if they differ by exactly 1 in exactly one coordinate, and are the same in every other coordinate.

    \subsection{Weighted graphs as metric spaces}

    The \defn{length} of a path $P$ in a weighted graph $G$ is the sum of the weights of the edges between consecutive vertices of $P$. The \defn{distance} in $G$ between $u,v\in V(G)$, denoted \defn{$\dist_G(u,v)$}, is the length of a shortest path from $u$ to $v$, or $\infty$ if no path exists. It is important to observe that if we take a subgraph $H$ of $G$, $\dist_H$ is not necessarily the restriction of $\dist_G$ to $V(H)\times V(H)$. This can be problematic, and subgraphs for which the distances are preserved are desirable; this motivates the next definition. A weighted graph $H$ with $V(H)\subseteq V(G)$ is \defn{isometric} in $G$ if $\dist_H(u,v)=\dist_G(u,v)$ for all $u,v\in V(H)$.

    The \defn{diameter} of a weighted graph $H$, denoted \defn{$\diam(H)$}, is the maximum distance between two vertices of $H$. The \defn{radius} of $H$, denoted \defn{$\rad(H)$}, is the smallest value $r$ for which there exists $u\in V(G)$ such that every other vertex is at distance at most $r$ from $u$. $u$ is then called a \defn{central vertex} of $H$.

    The \defn{distance} in a weighted graph $G$ between $u\in V(G)$ and $S\subseteq V(G)$, denoted \defn{$\dist_G(u,S)$}, is the minimum distance between $u$ and some $s\in S$. For a real number $r\geq 0$, the \defn{$r$-neighbourhood} of $S\subseteq V(G)$ in $G$, denoted \defn{$N_G^r(S)$}, is the set of vertices $v\in V(G)$ such that $\dist_G(v,S)\leq r$; note that $N_G^0(S)=S$. In particular, if $S$ is a singleton $\{s\}$, we instead write \defn{$N_G^r(s)$}. The \defn{weak diameter} of $S$ in $G$, denoted \defn{$\wdiam_G(S)$}, is the maximum distance between vertices of $S$ in $G$. We remark that ``weak diameter" is often referred to as ``diameter" in the metric space literature; we use ``weak diameter" to avoid confusion with the graph-theoretic notion of ``diameter".
    
    For a real number $r>0$, the \defn{$r$-th power} of a weighted graph $G$, denoted \defn{$G^r$}, is the graph with vertex set $V(G)$ where two vertices $u,v$ are adjacent if $\dist_G(u,v)\leq r$. An \defn{$r$-walk} [resp.\ \defn{$r$-path}] $P$ from $u\in V(G)$ to $v\in V(G)$ in $G$ is a sequence of vertices of $G$ that form a walk [resp.\ path] from $u$ to $v$ in $G^r$. This is equivalent to requiring that consecutive vertices are at distance at most $r$. For technical reasons, we permit walks, and thus $r$-walks, to have duplicate consecutive vertices. For $Z\subseteq V(G)$, an \defn{$r$-subwalk} [resp.\ \defn{$r$-subpath}] of $P-Z$ is a subsequence of $P$ consisting of a string of consecutive vertices of $P$ which are all not in $Z$. In particular, an \defn{$r$-subwalk} [resp.\ \defn{$r$-subpath}] of $P$ is any string of consecutive vertices in $P$. The \defn{interior} of $P$, denoted \defn{$\Int{P}$}, is the (possibly empty) $r$-subwalk [resp.\ $r$-subpath] of $P$ consisting of the entire sequence except the start $u$ and the end $v$. The vertices of $\Int(P)$ are called the \defn{interior vertices} of $P$. Note that if $\Int(P)$ has weak diameter in $G$ at most $d$, then $P$ has weak diameter in $G$ at most $d+2r$.
    
    \subsection{Colourings and monochromatic $r$-components}
    
    Given a graph $G$, a set of vertices $S\subseteq V(G)$, and an arbitrary set $C$, a \defn{colouring} of $S$ with \defn{colours} $C$ is a map $c:V(G)\rightarrow C$. We also say that a \defn{colouring} of $G$ with colours $C$ is a colouring of $V(G)$ with colours $C$. Generally, one is only concerned with the number of colours used, $|C|$, hence an \defn{$m$-colouring} is colouring on a set of colours $C$ with $|C|=m$, and it is customary to take $C:=\{1,\dots,m\}$. Note that for $n\in \{1,\dots,m-1\}$, an $n$-colouring can be considered to be an $m$-colouring by adding $m-n$ unused colours.

    Let $G$ be a graph, and let $S,T\subseteq V(G)$ with $S\subseteq T$. We say that a colouring $c_T$ of $T$ \defn{extends} a colouring $c_S$ of $S$ if $c_T\big|_S=c_S$. Notice that if $S_1,\dots,S_n\subseteq V(G)$, and $c_i$ is a colouring of $S_i$ for each $i\in \{1,\dots,n\}$ such that $c_i=c_j$ on $S_i\cap S_j$ for any $i,j\in \{1,\dots,n\}$, then there exists a colouring $c_G$ of $G$ that extends $c_1,\dots,c_n$ simultaneously. Furthermore, if $\bigcup_{i=1}^n S_i=V(G)$, then this colouring $c_G$ is unique, and denoted by \defn{$\bigcup_{i=1}^n c_i$}.
    
    Under a colouring $c$ of a graph $G$, for any colour $i$, a set of vertices $S\subseteq V(G)$ is said to be \defn{$i$-monochromatic} under $c$ if every vertex of $S$ received the colour $i$ under $c$. $S$ is also said to be \defn{monochromatic} under $c$ if it is $i$-monochromatic under $c$ for some colour $i$.
    
    For a weighted graph $G$ and a real number $r>0$, we say that an \defn{$i$-monochromatic $r$-component} of $G$ under $c$ is a maximal nonempty connected subgraph $M$ of $G^r$ such that $V(M)$ is $i$-monochromatic under $c$. We then say that a subgraph $M$ of $G^r$ is a \defn{monochromatic $r$-component} of $G$ under $c$ if $M$ is an $i$-monochromatic $r$-component of $G$ for some colour $i$. Observe that for any monochromatic $r$-component $M$ of $G$ and any $u,v\in V(M)$, there exists a monochromatic $r$-path from $u$ to $v$ in $G$ of the same colour, and that for any monochromatic $r$-path $P$ in $G$, there exists a monochromatic $r$-component $M$ of $G$ of the same colour such that $P\subseteq V(M)$.

    Given a weighted graph $G$, an integer $m\geq 1$, and real numbers $r>0$, $d\geq 0$, we say that a $m$-colouring $c$ of $G$ is an \defn{$(m,r,d)$-colouring} if for each monochromatic $r$-component $M$ of $G$ under $c$, $V(M)$ has weak diameter in $G$ at most $d$; note that this is equivalent to requiring that every monochromatic $r$-path in $G$ under $c$ has weak diameter in $G$ at most $d$. This definition is very similar to the concept of weak diameter colourings of $G^r$ used by \citet*{Bonamy}, with the key difference that the weak diameter is measured in $G$, rather than $G^r$; this is done to avoid having to constantly convert between distances in $G$ and distances in $G^r$. Observe that for any weighted graph $H$ that is isometric in $G$ and any $(m,r,d)$-colouring $c$ of $G$, $c\big|_{V(H)}$ is an $(m,r,d)$-colouring of $H$.

    \subsection{Control functions and dimension}
    \label{secasdim}

    Many of the following definitions are given in terms of weighted graphs, however we remark that they can be defined for (and sometimes originate from the study of) arbitrary metric spaces, with the noteworthy difference of ``diameter" (in the metric geometry sense) being used in place of ``weak diameter" whenever applicable.
    
    For a weighted graph $G$ and real number $r>0$, $S,T\subseteq V(G)$ are said to be \defn{$r$-disjoint} if for all $u\in S$, $v\in T$, we have $\dist_G(u,v)>r$. A collection $\scr{C}$ of subsets of $V(G)$ is said to be \defn{$r$-disjoint} if the sets in $\scr{C}$ are pairwise $r$-disjoint.

    Given a weighted graph $G$ and an integer $n\geq 0$, a function $f:\ds{R}^+\rightarrow \ds{R}^+$ is said to be an \defn{$n$-dimensional control function} for $G$ if, for every real number $r>0$, there exist $n+1$ collections $\scr{C}_1,\dots,\scr{C}_{n+1}$ of subsets of $V(G)$ such that:

    \begin{enumerate}[label=(\alph*)]
        \item $\bigcup_{i=1}^{n+1}\bigcup_{S\in \scr{C}_i}S = V(G)$;
        \item $\scr{C}_i$ is $r$-disjoint for each $i\in \{1,\dots,n+1\}$; and
        \item $\wdiam_G(S)\leq f(r)$ for each $i\in \{1,\dots,n+1\}$ and each $S\in \scr{C}_i$.
    \end{enumerate}

    Note that if $f$ is an $n$-dimensional control function for $G$, it is also an $m$-dimensional control function for $G$ for each integer $m\geq n$.

    This definition is the standard definition used for general metric spaces, however it is rather cumbersome. Thankfully, given that we are only interested in weighted graphs, we obtain an alternative definition as a consequence of the following result, whose proof is in the appendix.
    
    \begin{restatable}{proposition}{dimCol}
        \label{dimToColouring}
        Let $n\geq 0$ be an integer, and let $G$ be a weighted graph. Then $f:\ds{R}^+\rightarrow \ds{R}^+$ is an $n$-dimensional control function for $G$ if and only if, for every real number $r>0$, $G$ admits an $(n+1,r,f(r))$-colouring.
    \end{restatable}

    We use this as the definition of a control function from now on. We remark that \cref{dimToColouring} is very similar to Proposition~1.17 of \citet*{Bonamy}, with the key difference that Proposition~1.17 of \citet*{Bonamy} is stated in terms of asymptotic dimension (see below) instead of control functions. Furthermore, following the proof, one does not actually get an equivalence between control functions and the bound on the weak diameter of vertex sets of monochromatic $r$-components, instead picking up an extra factor of $r$. 

    The \defn{asymptotic dimension} of a weighted graph $G$, denoted \defn{$\asdim(G)$}, is the smallest integer $n\geq 0$ such that $G$ admits an $n$-dimensional control function, or $\infty$ otherwise. While this is consistent with how asymptotic dimension is defined on arbitrary metric spaces, we remark that this definition is not very interesting, since every finite weighted graph $G$ has asymptotic dimension 0. Thus, we extend the definitions of control function and asymptotic dimension to classes of graphs.

    Given an integer $n\geq 0$ and a class of weighted graphs $\scr{G}$, we say that a function $f:\ds{R}^+\rightarrow \ds{R}^+$ is an \defn{$n$-dimensional control function for $\scr{G}$} if $f$ is an $n$-dimensional control function for each $G\in \scr{G}$. The \defn{asymptotic dimension} of $\scr{G}$, denoted \defn{$\asdim(\scr{G})$}, is then the smallest integer $n\geq 0$ such that $\scr{G}$ admits an $n$-dimensional control function, or $\infty$ otherwise.

    Finding the asymptotic dimension of an infinite weighted graph is non-trivial. Nevertheless, the following result shows a natural way to reduce this problem to a question about finite graphs.

    \begin{restatable}{proposition}{ANdimInf}
        \label{ANdimInfinite}
        Let $n\geq 0$ be an integer and let $G$ be an infinite weighted graph. Let $\scr{A}$ be the class of all finite induced subgraphs of $G$, and let $f$ be an $n$-dimensional control function for $\scr{A}$. Then for any real number $\eps>0$, $r\mapsto f((1+\eps)r)$ is an $n$-dimensional control function for $G$.
    \end{restatable}

    The proof of \cref{ANdimInfinite} is also in the appendix.

    It is because of \cref{ANdimInfinite} that we need only consider finite graphs in this paper. We also remark that \cref{ANdimInfinite} is nearly identical to Proposition~A$.2$ of \citet*{Bonamy}, albeit with a slight tweaking to the control function for $G$; this will be important once we introduce Assouad--Nagata dimension.

    A function $f:\ds{R}^+\rightarrow \ds{R}^+$ is a \defn{dilation} if there exists a real number $c>0$ such that $f(r)\leq cr$ for every real number $r>0$. The \defn{Assouad--Nagata dimension} of a graph $G$ is the smallest integer $n\geq 0$ such that $G$ admits an $n$-dimensional control function $f$ that is also a dilation, or $\infty$ otherwise. Likewise, the \defn{Assouad--Nagata dimension} of a class of graphs $\scr{G}$ is the smallest integer $n\geq 0$ such that $\scr{G}$ admits an $n$-dimensional control function $f$ that is also a dilation, or $\infty$ otherwise. Observe that the asymptotic dimension of a class is at most the Assouad--Nagata dimension.

    Notice that, in \cref{ANdimInfinite}, if $f$ is a dilation, then so is $r\mapsto f((1+\eps)r)$. Thus, we can also use \cref{ANdimInfinite} to bound the Assouad--Nagata dimension of an infinite graph in terms of the Assouad--Nagata dimension of the class of its finite induced subgraphs. Thus, it still makes sense to only consider finite graphs even when discussing Assouad--Nagata dimension rather than asymptotic dimension.

    Finally, given a function $f:\ds{R}^+\rightarrow \ds{R}^+$, a weighted graph $G$, and an integer $n\geq 0$, sometimes we can only find $(n+1,r,f(r))$-colourings of $G$ when $r$ is sufficiently large. As such, for a real number $\ell\geq 0$, we say that $f$ is an \defn{$\ell$-almost $n$-dimensional control function} for a $G$ if $G$ admits an $(n+1,r,f(r))$-colouring for every real number $r>0$ with $r\geq \ell$. Note that a 0-almost $n$-dimensional control function is just a $n$-dimensional control function.

    \section{Path Rerouting}
	
    One of the key ideas of this paper is the ability to ``adjust" or ``reroute" particular $r$-paths to obtain an $r'$-path with the same endpoints but whose interior is contained within a ``better" part of the graph $G$. This is formalised in \cref{rerouting}. Note that this concept of rerouting is not completely new, and exists in some capacity within Lemma~3.2 of \citet*{Bonamy} without being formally identified.
	
    \begin{proposition}
	    \label{rerouting}
	    Let $r>0$, $\ell,d\geq 0$ be real numbers, let $G$ be a weighted graph, let $x,y\in V(G)$, and let $P$ be an $r$-path in $G$ from $x$ to $y$. Assume there exist sets $Z,S\subseteq V(G)$ and a map $\iota:Z\rightarrow S$ such that:
        \begin{enumerate}[label=(\alph*)]
            \item $\dist_G(v,\iota(v))\leq \ell$ for each $v\in Z$; and
            \item every $r$-subpath of $P-Z$ has weak diameter in $G$ at most $d$.
        \end{enumerate}
        
        Let $v_1,\dots,v_n$ be the (possibly empty) subsequence of $P$ that is in $Z$, including $x$,$y$ if applicable. Define $P'$ to be the sequence $x,\iota(v_1),\iota(v_2),\dots,\iota(v_n),y$. Then $P'$ is a $(d+2r+2\ell)$-walk from $x$ to $y$ in $G$ whose interior is contained in $S$.
    \end{proposition}
    
    \begin{proof}
        It suffices to show that consecutive vertices of $P'$ are at distance at most $d+2r+2\ell$. If $P$ does not intersect $Z$, $P$ is an $r$-subpath of $P-Z$, and hence $\dist_G(x,y)\leq d\leq d+2r+2\ell$ as desired. So we may assume that $P$ intersects $Z$, and that $v_1,\dots,v_n$ are defined with $n\geq 1$.
    
        Consider consecutive interior vertices $\iota(v_i),\iota(v_{i+1})$ of $P'$; we must show that $\dist_G(\iota(v_i),\iota(v_{i+1}))\leq d+2r+2\ell$; it suffices to show $\dist_G(v_i,v_{i+1})\leq d+2r$. If $v_i,v_{i+1}$ are consecutive in $P$, then $\dist_G(v_i,v_{i+1})\leq r$, as desired. Otherwise, let $u_i$ be the vertex directly after $v_i$ in $P$, and let $w_{i+1}$ be the vertex directly before $v_{i+1}$ in $P$; it suffices to show that $\dist_G(u_i,w_{i+1})\leq d$. Consider the $r$-subpath $P_i$ of $P$ going from $u_i$ to $w_{i+1}$. $P_i\cap Z=\emptyset$ by definition of the subsequence $v_1,\dots,v_n$, thus $P_i$ is an $r$-subpath of $P-Z$ and has weak diameter in $G$ at most $d$. Therefore, $\dist_G(u_i,w_{i+1})\leq d$, as desired.
     
        It remains to show that $\dist_G(x,\iota(v_1))$ and $\dist_G(y,\iota(v_n))$ are at most $d+2r+2\ell$. We show that $\dist_G(x,\iota(v_1))\leq d+2r+2\ell$; the argument for $\dist_G(y,\iota(v_n))$ is symmetric. If $x\in Z$, then $v_1=x$ and thus $\dist_G(x,\iota(v_1))=\dist_G(x,\iota(x))\leq \ell\leq d+2r+2\ell$. Otherwise, let $w_1$ be the vertex directly before $v_1$ in $P$; it suffices to show that $\dist_G(x,w_1)\leq d+r+\ell$. Let $P'$ be the $r$-subpath of $P$ going from $x$ to $w_1$; we must have $P'\cap Z=\emptyset$ by definition of $v_1$, hence $P'$ is an $r$-subpath of $P-Z$. Thus, $\wdiam_G(P')\leq d$; in particular $\dist_G(x,w_1)\leq d\leq d+2r+2\ell$, as desired.
    \end{proof}
	
    An important application for \cref{rerouting} is when the set $S$ is small. To help express this, we borrow an idea from \citet*{Bonamy}. Given a weighted graph $G$, $Z\subseteq V(G)$ is said to be a $(k,\ell)$-centred set in $G$ if there exists a set $S\subseteq V(G)$ with $|S|\leq k$ such that $Z\subseteq N_G^{\ell}(S)$. We call $S$ a \defn{centre} of $Z$. Note that for any weighted graph $G$ of radius at most $\ell$, $V(G)$ is a $(1,\ell)$-centred set in $G$, and any central vertex of $G$ is a centre of $V(G)$.

    \begin{corollary}
	\label{rpathInsideCentred}
        Let $k\geq 1$ be an integer, let $r>0$, $\ell\geq 0$ be real numbers, let $G$ be a weighted graph, let $Z$ be a $(k,\ell)$-centred set in $G$, and let $x,y\in Z$ be such that there exists an $r$-path $P$ from $x$ to $y$ in $G$ that is contained in $Z$. Then $\dist_G(x,y)\leq (k+1)(2r+2\ell)$.
    \end{corollary}
	
    \begin{proof}
	    Let $S$ denote a centre of $Z$, and define $\iota:Z\rightarrow S$ such that $\iota(v)\in N_G^{\ell}(v)\cap S$ for each $v\in Z$. Since $P\subseteq Z$, it is immediate that every $r$-subpath of $P-Z$ has weak diameter in $G$ at most $0$. We also know that for any $v\in Z$, $\dist_G(v,\iota(v))\leq \ell$. Thus, we can apply \cref{rerouting} to obtain a $(2r+2\ell)$-walk from $x$ to $y$ whose interior is in $S$. Since $|S|\leq k$, $P'$ can involve at most $k+2$ distinct vertices, and thus $\dist_G(x,y)\leq (k+1)(2r+2\ell)$, as desired.
    \end{proof}
    
    The next result is very similar to Lemma~2.1 of \citet*{Bonamy}, albeit the original's bound as stated is not tight enough for our purposes. We instead use \cref{rerouting} for a different, quicker, proof that avoids this problem.
	
    \begin{corollary}
	    \label{colouringMinusCentredSet}
	    Let $G$ be a weighted graph, let $r>0$, $d,\ell\geq 0$ be real numbers, let $m\geq 1$ be an integer, let $C$ be a set of colours of size $m$, and let $Z\subseteq V(G)$ be a $(k,\ell)$-centred set in $G$ such that $G-Z$ admits an $(m,r,d)$-colouring $c$ with colours $C$. Then for any colouring $c'$ of $Z$ with colours $C$, the colouring $c\cup c'$ of $G$ is a $(m,r,(k+1)(d+4r+2\ell))$-colouring with colours $C$.
    \end{corollary}
	
    \begin{proof}
        Let $S$ denote a centre of $Z$, and let $Z':=N_G^{r+\ell}(S)$. For any $u,v\in V(G)\setminus Z'$ at distance at most $r$ in $G$ and any shortest path $Q$ from $u$ to $v$ in $G$, observe that $Q$ is disjoint from $Z$, as otherwise we would have $u,v\in N_G^r(Z)\subseteq Z'$. Thus, $Q$ is also a path of length at most $r$ in $G-Z$, and $\dist_{G-Z}(u,v)\leq r$. Therefore, for any monochromatic $r$-path $P$ in $G$ under $c\cup c'$, any $r$-subpath $P'$ of $P-Z'$ is a monochromatic $r$-path in $G-Z$ under $c$. Thus, $P'$ has weak diameter in $G-Z$ (and $G$) at most $d$. Define a map $\iota:Z'\rightarrow S$ by selecting $\iota(z)\in N_G^{r+\ell}(z)\cap S$ for each $z\in Z'$; by \cref{rerouting}, we obtain a $(d+2r+2(r+\ell))=(d+4r+2\ell)$-walk which has the same endpoints as $P$ but whose interior is contained in $S$. Since $|S|\leq k$, this $(d+4r+2\ell)$-walk involves at most $k+2$ distinct vertices, and thus has weak diameter at most $(k+1)(d+4r+2\ell)$ in $G$, as desired.
    \end{proof}

    The main purpose of \cref{rerouting,rpathInsideCentred,colouringMinusCentredSet} is to give sufficient conditions under which we can show that a given colouring is a $(m,r,D)$-colouring. With this in mind, we introduce one more tool to do this in a non-trivial setting, however we need to introduce a bit more terminology first.

    A \defn{separation} of a graph $G$ is a pair of subgraphs \defn{$(A,B)$} of $G$ such that $G=A\cup B$. The set $S:=V(A\cap B)$ is called the \defn{separator} of the separation. 
    
    Next, given a real number $r>0$, a set of colours $C$ of size at least 2, a weighted graph $G$, $S\subseteq V(G)$, and $Z\supseteq N_G^{3r}(S)$, we say that a colouring $c$ of $Z$ with colours $C$ has an \defn{$(S,r)$-barrier} in $G$ if there exist distinct colours $\alpha,\beta \in C$ such that:
    \begin{enumerate}[label=(\alph*)]
        \item for each $v\in N_G^{r}(S)\setminus S$, there exists $s\in N_G^r(v)\cap S$ with $c(s)=c(v)$
        \item $N_G^{2r}(S)\setminus N_G^{r}(S)$ and $N_G^{3r}(S)\setminus N_G^{2r}(S)$ are $\alpha$-monochromatic and $\beta$-monochromatic respectively under $c$. 
    \end{enumerate}

    Note that this definition is not completely new; something similar was used in the proof of Lemma~3.2 in \citet*{Bonamy} without being formally identified. Note that for any colouring $c_S$ of $S$ with colours $C$, there exists a colouring of $Z$ with colours $C$ that extends $c_S$ and admits an $(S,r)$-barrier. Further, note that the presence of an $(S,r)$-barrier in $G$ is preserved upon extending a colouring.
    
    The choice of name ``barrier" comes from the following observation. For any $r$-path $P$ from $u\in N_G^{2r}(S)$ to $v\in V(G)\setminus N_G^{2r}(S)$, $P$ must contain a vertex in both $N_G^{2r}(S)\setminus N_G^{r}(S)$ and $N_G^{3r}(S)\setminus N_G^{2r}(S)$. Since these two sets are monochromatic of different colours, this means that $P$ cannot be monochromatic. Hence, this colouring contains a ``barrier" near $S$ that blocks monochromatic $r$-paths. The requirement on how $N_G^{r}(S)-S$ is coloured is to prevent a ``sudden" change in colouring, in the sense that vertices close to $S$ still use similar colours to $S$, and the monochromatic sets only appear further away from $S$. This is not important when $G$ is the full graph, however we use it when $G$ is one half of a separation of a larger graph, and $S$ is the separator. Property~(a) then prevents a sudden transition of colours upon crossing the separator, and property~(b) prevents monochromatic $r$-components from leaking far across the separator.

    We are now in position to state the main tool of this section, \cref{boundComponentsDiameter}. While the statement of \cref{boundComponentsDiameter} is new, the proof uses similar techniques to the later parts of the proof of Lemma~3.2 in \citet*{Bonamy}.

    \begin{lemma}
        \label{boundComponentsDiameter}
        Let $r>0$, $\ell\geq 0$ be real numbers, let $k\geq 1$ be an integer, and set $\ell':=(k+1)(6r+2\ell)$ and $r':=2r+2\ell'$. Let $m\geq 2$ be an integer, let $C$ be a set of colours of size $m$, and let $d,D>0$ be real numbers such that $D\geq d+2r'$. Let $G$ be a weighted graph, and for some integer $a\geq 0$, let $G_0,G_1,\dots,G_a$ be isometric subgraphs of $G$ such that:
        \begin{enumerate}[label=(\alph*)]
            \item $G=\bigcup_{i=0}^a G_i$; 
            \item for each $i\in \{1,\dots,a\}$, if $\widetilde{G}_i:=\bigcup_{j=0,j\neq i}^a G_j$ and $S_i:=V(G_i\cap \widetilde{G}_i)$, then:
            \begin{enumerate}[label=(\roman*)]
                \item $S_i\subseteq V(G_0)$; and
                \item $S_i$ is a $(k,\ell)$-centred set in $G$;
            \end{enumerate}
            \item $G_0$ admits a $(m,r',d)$-colouring $c_0$ with colours $C$;
            \item for each $i\in \{1,\dots,a\}$, $G_i$ admits a colouring $c_i$ with colours $C$ such that:
            \begin{enumerate}[label=(\roman*)]
                \item $c_i$ is an $(m,r,D)$-colouring of $G_i$;
                \item $c_i=c_0$ on $S_i$; and
                \item $c_i$ has an $(S_i,r)$-barrier.
            \end{enumerate}
        \end{enumerate}

        Then $c:=\bigcup_{i=0}^a c_i$ is a well-defined $(m,r,D)$-colouring of $G$ with colours $C$.
    \end{lemma}

    \begin{proof}
        If $a=0$ the claim is trivially true, as $G=G_0$ by (a) and $c=c_0$ is an $(m,r,D)$-colouring of $G_0=G$ by (c), since $r'\geq r$ and $d\leq D$. So assume $a\geq 1$.

        Observe that $c$ is a well-defined colouring of $G$ by (a), (b)(i) and (d)(ii). So it suffices to show that for every real number $r>0$ and every monochromatic $r$-component $M$ of $G$ under $c$, $\wdiam_G(V(M))\leq D$. 
        
        By (a), for each $i\in \{1,\dots,a\}$, $(G_i,\widetilde{G}_i)$ is a separation of $G$ with separator $S_i$, which we use implicitly throughout this proof. Also, note that $\bigcap_{i=1}^a V(\widetilde{G}_i)=V(G_0)$ by (b)(i).
    
        For all $i\in \{1,\dots,a\}$ and $j\in \{1,2,3\}$, let $N_i^j:=N_{G_i}^{jr}(S_i)$ and $S_i^j:=N_{G_i}^{jr}(S_i)\setminus N_{G_i}^{(j-1)r}$. Note that for each $i\in \{1,\dots,a\}$, since $c_i$ has an $(S_i,r)$-barrier, $S_i^2$, $S_i^3$ are both monochromatic of different colours, and for each $v\in S_i^1$ there is an $s\in S_i\cap N_{G_i}^r(v)\subseteq S_i\cap N_G^r(v)$ with the same colour as $v$. Let $V':=V(G_0)\cup \bigcup_{i=1}^a N_i^2$.
    	    
        \textbf{Claim:} For each monochromatic $r$-component $M$ of $G$ under $c$, $V(M)$ is contained either in $V(G_i)$ for some $i\in \{1,\dots,a\}$, or in $V'$.

        \begin{proofofclaim}
            For any $i\in \{1,\dots,a\}$, any $r$-path $P$ from $u\in V(\widetilde{G}_i)\cup N_i^2$ to $v\in V(G_i)\setminus N_i^2$ must contain a vertex in $S_i^2$ and $S_i^3$; see \cref{FigRSep}. As $S_i^2$ and $S_i^3$ are monochromatic of different colours, $P$ cannot be monochromatic. It follows that for any monochromatic $r$-component $M$ of $G$, $V(M)$ must be contained in either $V(\widetilde{G}_i)\cup N_i^2$ or $V(G_i)\setminus N_i^2$. Accounting for the fact that this holds for each $i\in \{1,\dots,a\}$, we can conclude that $V(M)$ is either contained in $V(G_i)\setminus N_i^2\subseteq V(G_i)$ for some $i\in \{1,\dots,a\}$, or in $\bigcap_{i=1}^a (V(\widetilde{G}_i)\cup N_i^2) = V'$.
        \end{proofofclaim}
\begin{figure}[!ht]
\begin{center}
\begin{minipage}{0.85\textwidth}
\centering
\includegraphics[width=\textwidth]{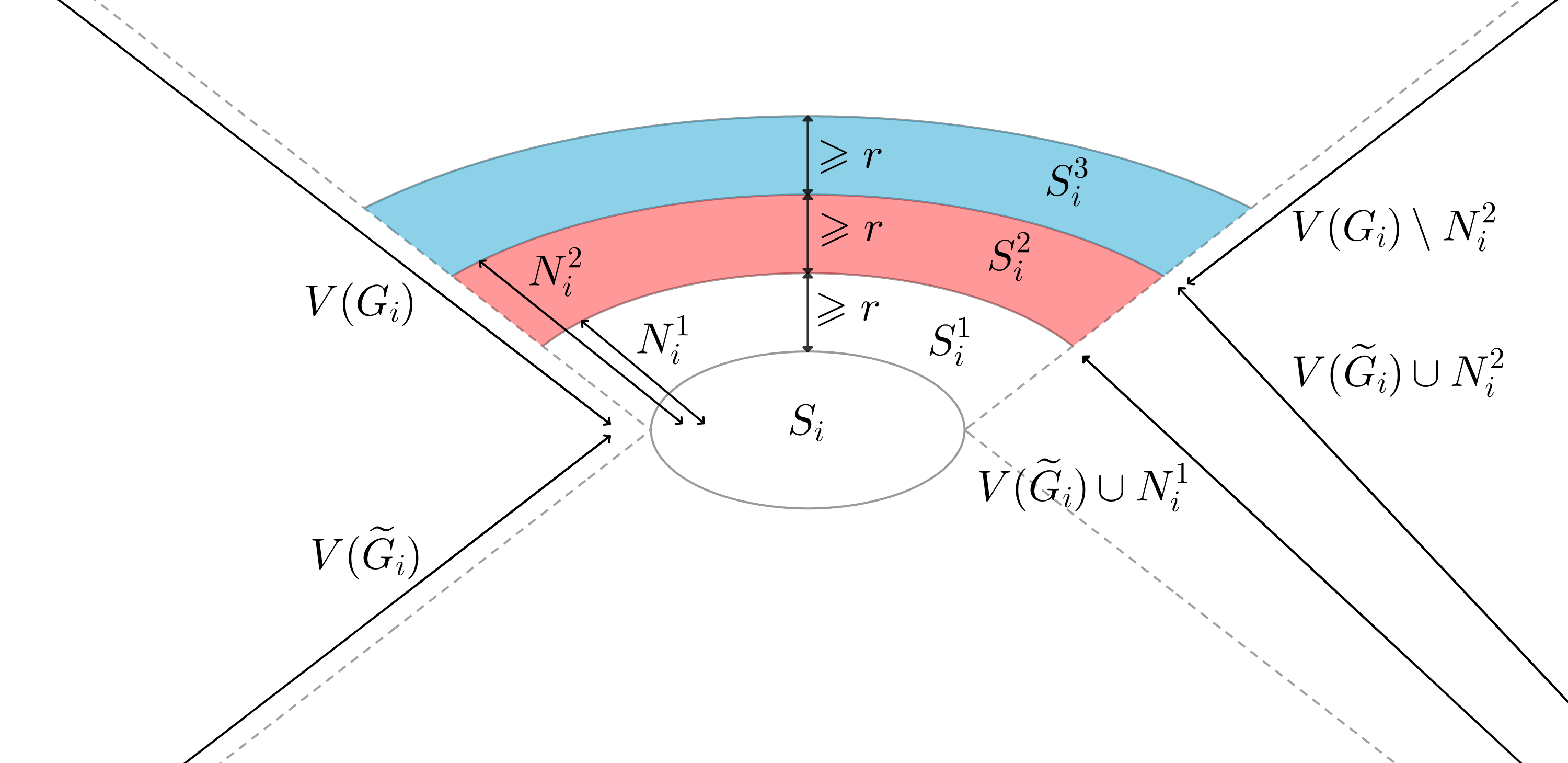}
\caption{\label{FigRSep}A diagram of the separation $(G_i,\widetilde{G}_i)$ of $G$ for some $i\in \{1,\dots,a\}$, with several relevant sets of vertices near the separator $S_i$ labelled. Vertices and edges of $G$ are not depicted, instead the dotted line denotes a region in which all vertices and edges are contained. Observe that any $r$-path $P$ from a vertex in $V(G_i)\setminus N_i^2$ to a vertex in $V(\widetilde{G}_i)\cup N_i^2$ must contain a vertex in both $S_i^2$ and $S_i^3$, which are monochromatic of different colours, depicted here as red and blue. Thus, $P$ cannot be monochromatic. Additionally, observe that any $r$-path from a vertex in $S_i^2$ to a vertex in $\widetilde{G}_i\cup N_i^1$ must contain a vertex in $S_i^1$.}
\end{minipage}
\end{center}
\end{figure}
        By assumption, $G_0,\dots,G_a$ are all isometric subgraphs of $G$. Thus, any monochromatic $r$-component $M$ of $G$ under $c$ whose vertices are contained in $V(G_i)$ for some $i\in \{1,\dots,a\}$ is also a monochromatic $r$-component of $G_i$ under $c_i$; in this case $V(M)$ has weak diameter at most $D$ in $G_i$ and $G$. Similarly, any monochromatic $r$-component $M$ of $G$ under $c$ whose vertices are contained in $V(G_0)$ is a subgraph of a monochromatic $r'$-component of $G_0$ under $c_0$ as $r\leq r'$, and thus $V(M)$ has weak diameter at most $d\leq D$ in $G_0$ and $G$. So henceforth, we only need to consider monochromatic $r$-components whose vertices are contained in $V'$ but are not contained in any $V(G_i)$, $i\in \{0,\dots,a\}$. Consider such a monochromatic $r$-component $M$.
    	    
        \textbf{Claim:} For each $v\in V(M)$, there exists $s\in V(G_0)\cap N_G^{\ell'}(v)$ with the same colour as $v$ under $c$.
    	\begin{proofofclaim}
            This is trivially true for any $v\in V(G_0)$, and for any $v\in S_i^1$ for some $i\in \{1,\dots,a\}$, the claim follows since $c_i$ has an $(S_i,r)$-barrier and $r\leq \ell'$. Since $V(M)\subseteq V'$, we now only need to consider $v\in S_i^2$ for some $i\in \{1,\dots,a\}$. Observe that $V(M)$ is not contained in $S_i^2$, as it is not contained in $V(G_i)$, thus there must exist some $u\in V(M)\setminus S_i^2$; since $V(M)\subseteq V'$, this means that $u\in V(\widetilde{G}_i)\cup N_i^1$. Because $u,v$ are both in $M$, there exists a monochromatic $r$-path $P$ from $v\in S_i^2$ to $u\in V(\widetilde{G}_i)\cup N_i^1$ in $G$ of the same colour. Observe that $P$ must contain a vertex in $S_i^1$; see \cref{FigRSep}. Let $x$ be the first (in the order of the sequence of $P$) such vertex in $S_i^1\cap P$; observe that every vertex before $x$ must be in $S_i^2$. Thus, the $r$-subpath $P'$ of $P$ from $v$ to $x$ is an $r$-path contained in $S_i^1\cup S_i^2\subseteq N_i^2$. Since $c_i$ has an $(S_i,r)$-barrier, there must be an $s\in S_i\cap N_G^r(x)\subseteq V(G_0)$ with the same colour as $x$. Hence $v$ and $s$ have the same colour, and there is an $r$-path contained in $N_i^2$ from $v$ to $s$ obtained by appending $s$ to the end of $P'$. Since $S_i$ is a $(k,\ell)$-centred set in $G$, $N_i^2$ is a $(k,\ell+2r)$-centred set in $G$. Thus, we can apply \cref{rpathInsideCentred} to find that $\dist_G(v,s)\leq (k+1)(2r+2(2r+\ell))=\ell'$, as desired.
    	\end{proofofclaim}
    	    
        Therefore, we may define a map $\iota:V(M)\rightarrow S$ that sends $v\in V(M)$ to some $s\in V(G_0)\cap N_G^{\ell'}(v)$ with the same colour as $v$. Now, for any two vertices $x,y$ in $V(M)$, we know that there exists a monochromatic $r$-path $P$ from $x$ to $y$ in $G$ whose vertices are contained in $V(M)$. Thus, every $r$-subpath of $P-V(M)$ has weak diameter in $G$ at most 0. By applying \cref{rerouting}, we obtain a monochromatic $(0+2r+2\ell')=r'$-path $P'$ in $G$ of the same colour as $P$ but whose interior is contained in $V(G_0)$. As $G_0$ is an isometric subgraph, $\Int(P')$ is a monochromatic $r'$-path in $G_0$ under $c_0$, and thus has weak diameter in $G_0$ and $G$ at most $d$. Hence, $P$ has weak diameter in $G$ at most $d+2r'$; in particular $\dist_G(x,y)\leq d+2r'\leq D$, as desired.
    \end{proof}

    \section{Colourings on Tree-Decompositions}

    Given graphs $G,H$, an \defn{$H$-decomposition} of $G$ is a collection $(B_{x}:x\in V(H))$ of subsets of $V(G)$ such that (a) for each $v\in V(G)$, the subgraph of $H$ induced by the vertices $x\in V(H)$ with $v\in B_x$ is nonempty and connected; and
    (b) for each $uv\in E(G)$, there exists $x\in V(H)$ such that $u,v\in B_x$. The subsets $B_x$, $x\in V(H)$, are called \defn{bags}. The \defn{width} of an $H$-decomposition is $\max_{x\in V(H)}(|B_x|-1)$. The \defn{adhesion} of an $H$-decomposition is $\max_{xy\in E(H)}(|B_x\cap B_y|)$. If $V(H)\subseteq V(G)$, $(B_x:x\in V(H))$ is \defn{rooted} if for each $x\in V(H)$, $x\in B_x$. An $H$-decomposition of $G$ is a \defn{tree-decomposition} of $G$ if $H$ is a tree $T$; the \defn{treewidth} of $G$ is then the minimum width of a tree-decomposition of $G$. Note that 
    every graph of treewidth at most $k$ has a tree-decomposition of width and adhesion at most $k$.

    Given a tree-decomposition $(B_t:t\in V(T))$ of a graph $G$ and a vertex $t\in V(T)$, the \defn{torso} of $G$ at $t$ with respect to $(B_t:t\in V(T))$, denoted \defn{$\UWT{G}{B_t}$}, is the graph obtained from $G[B_t]$ by adding an edge $uv$ whenever there exists $t'\in V(T)$ adjacent to $t$ such that $u,v\in B_t\cap B_{t'}$, provided $uv$ does not already exist in $G[B_t]$. This is a standard definition, which we now extend. If $G$ is a weighted graph, we define the \defn{weighted torso} of $G$ at $t$ with respect to $(B_t:t\in V(T))$, denoted \defn{$\WT{G}{B_t}$}, to be the weighted graph obtained by imbuing $\UWT{G}{B_t}$ with the weighting $w$ defined by $w(uv):=\dist_G(u,v)$ for each $uv\in E(\UWT{G}{B_t})$. We emphasise that the weight of $uv$ is determined by the distance between $u$ and $v$ in the whole graph $G$, not merely in $G[B_t]$, and that even when $G$ is unweighted, the weighted torso is usually distinct from the torso, with edge weights other than $1$.
    
    For any subtree $T'$ of $T$, consider $\widehat{G}_{T'}:=\bigcup_{t\in V(T')}\WT{G}{B_t}$; note that this is a well-defined weighted graph, as the weightings on different weighted torsos agree. Observe that shortest paths in $G$ can only exit $V(\widehat{G}_{T'})$ via a vertex in $B_t\cap B_{t'}$ for some $tt'\in E(T)$ with $t\in V(T')$, $t'\in V(T)\setminus V(T')$, and must also reenter $V(\widehat{G}_{T'})$ via another vertex in $B_t\cap B_{t'}$; at this point, the extra edges added to $G[B_t]$ to form $\WT{G}{B_t}$ provide a shortcut of the same length in $\widehat{G}_{T'}$. It follows that $\widehat{G}_{T'}$ is isometric in $G$. In particular, taking $V(T'):=\{t\}$ shows that the weighted torso $\WT{G}{B_t}$ is isometric in $G$, this is our main motivation for considering weighted torsos. Furthermore, we call $\widehat{G}_{T}$ the \defn{completion} of $G$ with respect to $(B_t:t\in V(T))$, and denote it simply as \defn{$\widehat{G}$}. By the previous observation, $\widehat{G}$ is isometric in $G$; in particular, $\dist_G=\dist_{\widehat{G}}$. As all edges of $\widehat{G}$ had weight $\dist_G$, this means that the weighting on $\widehat{G}$ is $\dist_{\widehat{G}}$. Also notice that $(B_t:t\in V(T))$ is still a tree-decomposition for $\widehat{G}$, and that for each $t\in V(T)$, $\UWT{\widehat{G}}{B_t}=\UNW(\widehat{G}[B_t])=\UWT{G}{B_t}$; because the weighting on $\widehat{G}$ is $\dist_{\widehat{G}}$, this gives $\WT{\widehat{G}}{B_t}=\widehat{G}[B_t]=\WT{G}{B_t}$. It follows that $\widehat{(\widehat{G})}=\widehat{G}$, hence the name ``completion". We also know that for each subtree $T'$ of $T$, $\widehat{(\widehat{G})}_{T'}=\bigcup_{t\in V(T')}\widehat{G}[B_t]=\widehat{G}_{T'}$ is isometric in $G$ and hence also $\widehat{G}$; this is critical to later proofs.

    For an integer $k\geq 0$ and a class of weighted graphs $\scr{H}$, we say that a tree-decomposition $(B_t:t\in V(T))$ of a weighted graph $G$ is a \defn{$(k,\scr{H})$-construction} of $G$ if it has adhesion at most $k$ and, for each $t\in V(T)$, $\WT{G}{B_t}\in \scr{H}$. If $G$ admits a $(k,\scr{H})$-construction, we say that $G$ is \defn{$(k,\scr{H})$-constructable}.

    A \defn{partition} $\scr{P}$ of a weighted graph $G$ is a collection of nonempty pairwise disjoint subsets of $V(G)$ such that $\bigcup_{P\in \scr{P}} P=V(G)$ and each $P\in \scr{P}$ induces a connected subgraph of $G$. The subsets $P\in \scr{P}$ are called the \defn{parts} of $\scr{P}$; for a vertex $v\in V(G)$, we use \defn{$\Part_{\scr{P}}(v)$} to denote the unique part $P\in \scr{P}$ for which $v\in P$, and for a set $S\subseteq V(G)$, we use \defn{$\Parts_{\scr{P}}(S)$} to denote the set of all parts $P\in \scr{P}$ for which $S\cap P\neq \emptyset$. The \defn{quotient} of $\scr{P}$ in $G$, denoted \defn{$G/\scr{P}$}, is the graph with vertex set $\scr{P}$ and an edge between two parts if the subgraphs of $G$ they induce are adjacent in $G$; note that $G/\scr{P}$ is a minor of $G$. For a real number $\ell\geq 0$, $\scr{P}$ is \defn{$\ell$-shallow} if, for each part $P\in \scr{P}$, $G[P]$ has radius at most $\ell$. If, for an integer $k\geq 0$, $G/\scr{P}$ also has treewidth at most $k$, then we say that $\scr{P}$ is a \defn{$(k,\ell)$-partition} of $G$.

    The main results of this section are as follows.

    \begin{theorem}
        \label{ANdimConstructable}
        Let $k\geq 0$ be an integer, let $\scr{H}$ be a class of weighted graphs, and let $\scr{G}$ be a hereditary class of weighted graphs that are all $(k,\scr{H})$-constructable. Then $\ANdim(\scr{G})\leq \max(\ANdim(\scr{H}),1)$
    \end{theorem}

    \begin{theorem}
        \label{ANdimPartition}
        For every integer $k\geq 0$, there exists a dilation $f_k$ such that for every real number $\ell\geq 0$, $f_k$ is an $\ell$-almost $1$-dimensional control function for every weighted graph that admits a $(k,\ell)$-partition.
    \end{theorem}

   For any graph $G$ of treewidth at most $k$, the partition into singletons is a $(k,0)$-partition of $G$; this gives the following as an easy application of \cref{ANdimPartition}.

    \begin{theorem}
        \label{ANdimTw}
        For any integer $k\geq 0$, any class of weighted graphs of treewidth at most $k$ has Assouad--Nagata dimension at most 1.
    \end{theorem}

    \cref{ANdimConstructable,ANdimPartition} are both special cases of a more general result, \cref{ANdimSC}, which we work towards stating now.

    For any weighted graph $G$, partition $\scr{P}$ of $G$, and tree-decomposition $(B_t:t\in V(T))$ of $G/\scr{P}$, observe that $(\bigcup_{P\in B_t}P:t\in V(T))$ is a tree-decomposition of $G$. For an integer $k\geq 0$, a real number $\ell\geq 0$, and a class of weighted graphs $\scr{H}$, we say that the pair $(\scr{P},(B_t:t\in V(T)))$ is a \defn{$(k,\ell,\scr{H})$-strong-construction} for $G$ if:

    \begin{enumerate}[label=(\alph*)]
        \item For each $t\in V(T)$ and $\scr{P}_t\subseteq B_t$, $\WT{G}{\bigcup_{P\in B_t}P}[\bigcup_{P\in \scr{P}_t}P]\in \scr{H}$;
        \item $\scr{P}$ is $\ell$-shallow; and
        \item $(B_t:t\in V(T))$ has adhesion at most $k$.
    \end{enumerate}

    We say that $G$ is \defn{$(k,\ell,\scr{H})$-strongly-constructable} if it admits a $(k,\ell,\scr{H})$-strong-construction.

    \begin{theorem}
        \label{ANdimSC}
        For every integer $k\geq 0$, and function $f:\ds{R}^+\rightarrow \ds{R}^+$, there exists a function $f_k:\ds{R}^+\rightarrow \ds{R}^+$ such that the following holds. Let $n\geq 0$ be an integer, let $\ell\geq 0$ be a real number, and let $\scr{H}$ be a class of graphs that admits $f$ as an $\ell$-almost $n$-dimensional control function. Set $n':=\max(n,1)$, then for every weighted graph $G$ that admits a $(k,\ell,\scr{H})$-strong-construction, $f_k$ is an $\ell$-almost $n'$-dimensional control function for $G$. Further, if $f$ is a dilation, then so is $f_k$.
    \end{theorem}

    Note that if $G$ is $(k,\scr{H})$-constructable and $\scr{H}$ is hereditary, then $G$ is $(k,0,\scr{H})$-strongly-constructable using the partition into singletons; \cref{ANdimConstructable} then follows. Separately, if $\scr{H}$ is the class of weighted graphs $H$ such that $V(H)$ is a $(k+1,\ell)$-centred set in $H$, then any weighted graph $G$ that admits a $(k,\ell)$-partition is $(k,\ell,\scr{H})$-strongly-constructable. Using \cref{rpathInsideCentred}, we find that, for any $r\geq \ell$ and any $H\in \scr{H}$, any $r$-path in $H$ has weak diameter in $H$ at most $4(k+2)r$; it follows any $1$-colouring of $H$ is a $(1,r,4(k+2)r)$-colouring. Thus, $r\mapsto 4(k+2)r$ is a dilation that is an $\ell$-almost $0$-dimensional control function for $\scr{H}$; \cref{ANdimPartition} follows. The remainder of this section is therefore dedicated to the proof of \cref{ANdimSC}. 

    \cref{ANdimSC} itself follows from the following, more technical, lemma.
    \begin{lemma}
    	\label{ANdimSCLemma}
            For every integer $k\geq 0$ and function $f:\ds{R}^+\rightarrow \ds{R}^+$, there exists a function $f_k:\ds{R}^+\rightarrow \ds{R}^+$ such that the following holds. Let $n\geq 0$ be an integer, let $\ell\geq 0$ be a real number, let $\scr{H}$ be a class of graphs for which $f$ is an $\ell$-almost $n$-dimensional control function, and let $G$ be a weighted graph with a $(k,\ell,\scr{H})$-strong-construction $(\scr{P}$,$(B_t:t\in V(T)))$. Set $n':=\max(n,1)$, let $r>0$ be a real number with $r\geq \ell$, let $q\in V(T)$, and let $S^{\scr{P}}\subseteq B_q$ be of size at most $k$. Set $S:=\bigcup_{P\in S^{\scr{P}}} P$, $Z^{\scr{P}}:=\Parts_{\scr{P}}(N_G^{3r}(S))$, and let $c_Z$ be an $(n'+1)$-colouring of $Z:=\bigcup_{P\in Z^{\scr{P}}} P$ with colours $\{1,\dots,n'+1\}$. Then $c_Z$ can be extended to an $(n'+1,r,f_k(r))$-colouring $c$ of $G$ with colours $\{1,\dots,n'+1\}$. Further, if $f$ is a dilation, then so is $f_k$.
    \end{lemma}
    	
    \begin{proof}
            Let $f_0:=f$, and for every integer $k\geq 1$ and real number $r>0$, let:
            \begin{align*}
                g_k'(r)&:=8(k+1)r,\\
                g_k^*(r)&:=2g_k'(r)+2r,\\
                f_k^*(r)&:=f_{k-1}(g_k^*(r)),\\
                f_k^{\#}(r)&:=(k+1)(f_k^*(r)+4g_k^*(r)+12r),\text{ and}\\
                f_k(r)&:=f_k^{\#}(r)+2g_k^*(r).
            \end{align*}
            Observe that if $f$ is a dilation, then for every integer $k\geq 0$, $g_k'$, $g_k^*$, $f_k^*$, $f_k^{\#}$, and $f_i$ are all dilations. Also, note that for every integer $k\geq 0$ and real number $r>0$ with $r\geq \ell$, $g_k^*(r)\geq r\geq \ell$.

            Now, let $r$ be defined as in the statement of \cref{ANdimSCLemma}; we need to show that $G$ admits an $(n'+1,r,f_k(r))$-colouring with colours $\{1,\dots,n'+1\}$ that extends $c_Z$. We do this via induction, primarily on $k$ and secondarily on $|V(G)|$.

            The base case occurs when $k=0$; in this case, for each connected component $C$ of $G$, $V(C)\subseteq \bigcup_{P\in B_t}P$ for some $t\in V(T)$. Furthermore, for this $t$, $C$ is isometric in $\WT{G}{\bigcup_{P\in B_t}P}$. Since $\WT{G}{\bigcup_{P\in B_t}P} \in \scr{H}$ and $r\geq \ell$, $\WT{G}{\bigcup_{P\in B_t}P}$ admits an $(n'+1,r,f(r))$-colouring, which we may assume is with colours $\{1,\dots,n'+1\}$; the restriction of this colouring to $V(C)$ then gives an $(n'+1,r,f(r))$-colouring of $C$ with colours $\{1,\dots,n'+1\}$. As this holds for every connected component of $G$, taking the union of these colourings gives an $(n'+1,r,f(r))$-colouring of $G$ with colours $\{1,\dots,n'+1\}$. Since $f(r)=f_0(r)$ and $Z=\emptyset$ as $|S|=0$, this is the desired colouring. So we may now assume that $k\geq 1$, and that the lemma holds for all smaller values of $k$ and when $k$ is the same but $|V(G)|$ is smaller.
            
            We begin by making a few assumptions. Notice that $(\bigcup_{P\in B_t}P : t\in V(T))$ is a tree-decomposition of $G$; let $\widehat{G}$ be the corresponding completion of $G$. We may assume that $G=\widehat{G}$. Otherwise, observe that $\scr{P}$ is still a partition of $\widehat{G}$, and that $(B_t:t\in V(T))$ is still a tree-decomposition for $\widehat{G}/\scr{P}$ of adhesion at most $k$, since the extra edges added to $G/\scr{P}$ to make $\widehat{G}/\scr{P}$ only go between parts that share a bag. Furthermore, note that edges of $G$ have weight in $\widehat{G}$ no larger than their weight in $G$; since no vertices or edges are deleted going from $G$ to $\widehat{G}$, this means for any $P\in \scr{P}$, $\rad(\widehat{G}[P])\leq \rad(G[P])\leq \ell$. Thus, $\scr{P}$ is also an $\ell$-shallow partition for $\widehat{G}$. Additionally, observe that $(\bigcup_{P\in B_t}P : t\in V(T))$ is still a tree-decomposition of $\widehat{G}$, and that the completion of $\widehat{G}$ with respect to $(\bigcup_{P\in B_t}P : t\in V(T))$ is still $\widehat{G}$. Also, for each $t\in V(T)$, observe that $\WT{\widehat{G}}{\bigcup_{P\in B_t}P}=\widehat{G}[\bigcup_{P\in B_t}P]=\WT{G}{\bigcup_{P\in B_t}P}\in \scr{H}$; thus for each $\scr{P}_t\subseteq B_t$, $\WT{\widehat{G}}{\bigcup_{P\in B_t}P}[\bigcup_{P\in \scr{P}_t}P]=\widehat{G}[\bigcup_{P\in \scr{P}_t}P]=\WT{G}{\bigcup_{P\in B_t}P}[\bigcup_{P\in \scr{P}_t}P]\in \scr{H}$. So $(\scr{P},(B_t:t\in V(T)))$ is still a $(k,\ell,\scr{H})$-strong-construction for $\widehat{G}$. Finally, as observed when we first defined the completion, we have $\dist_G=\dist_{\widehat{G}}$, hence $N_{\widehat{G}}^{3r}(S)=N_{G}^{3r}(S)$, and any $(n'+1,r,f_k(r))$-colouring of $\widehat{G}$ is an $(n'+1,r,f_k(r))$-colouring of $G$. Thus, we can proceed by setting $G:=\widehat{G}$.

            We may also assume that $S^{\scr{P}}$ is nonempty. Otherwise, if $G$ is empty, we are clearly done, and if $G$ is nonempty, pick any part $P'$ of $\scr{P}$ and pick some vertex $q'\in V(T)$ with $P'\in B_{q'}$. Let $S'^{\scr{P}}:=\{P\}$, let $S':=\bigcup_{P\in S'^{\scr{P}}} P=P'$, let ${Z^{\scr{P}}}'$ be the set of parts that intersect $N_G^{3r}(S')$, and let $c_{Z'}$ be an arbitrary $(n'+1)$-colouring of $Z':=\bigcup_{P\in {Z^{\scr{P}}}'} V(P)$ with colours $\{1,\dots,n'+1\}$. Observe that $Z$ must be empty, thus any colouring that extends $c_{Z'}$ also extends $c_{Z}$. Therefore, since $|S'^{\scr{P}}|=1\leq k$, we may proceed by setting $S^{\scr{P}}:=S'^{\scr{P}}$, $S:=S'$, $Z^{\scr{P}}:=Z'^{\scr{P}}$, $Z:=Z'$ and $c_{Z}:=c_{Z'}$. Observe also that since parts are nonempty, $S^{\scr{P}}$ being nonempty implies that $S$ and hence $Z$ are nonempty.

            Henceforth, we can now assume that $S^{\scr{P}}$, and consequently $Z$, are nonempty, and that $G=\widehat{G}$. We use the latter assumption implicitly throughout the remainder of the proof. Note that this gives, as mentioned when we argued that we could take $G=\widehat{G}$, that for each $t\in V(T)$ and each $\scr{P}_t\subseteq B_t$, $G[\bigcup_{P\in \scr{P}_t}P]\in \scr{H}$.

            For each $e=tt'\in E(T)$, let $S_e^{\scr{P}}:=B_t\cap B_{t'}$ and let $S_e:=\bigcup_{P\in S_e^{\scr{P}}}P$. Note that $|S_e^{\scr{P}}|\leq k$; consequently, $S_e$ is a $(k,r)$-centred set in $G$, as $\ell\leq r$.

            For each $P\in S^{\scr{P}}$, let $T_P$ be the subgraph of $T$ induced by the vertices $t\in V(T)$ such that $B_t\cap \Parts_{\scr{P}}(N_G^{3r}(P))\neq \emptyset$. Since $P$ induces a nonempty connected subgraph of $G$, $N_G^{3r}(P)$ also induces a nonempty connected subgraph of $G$, and thus $\Parts_{\scr{P}}(N_G^{3r}(P))$ induces a nonempty connected subgraph of $G/\scr{P}$. Hence, $T_P$ is nonempty and connected, and for each $tt'\in E(T_P)$, $B_t\cap B_{t'}\cap \Parts_{\scr{P}}(N_G^{3r}(P))\neq \emptyset$. Additionally, $T_P$ contains $q$ as a vertex, as $P\in S^{\scr{P}}\subseteq B_q$.
        
            Let $T':=\bigcup_{P\in S^{\scr{P}}}T_P$. Notice that $T'$ is connected, as each $T_P$, $P\in S^{\scr{P}}$, is connected and contains $q$, and nonempty, as $|S^{\scr{P}}|>0$. Further, notice that $t\in V(T)$ is in $V(T')$ if and only if $B_t\cap Z^{\scr{P}}\neq \emptyset$, and that $B_t\cap B_{t'}\cap Z^\scr{P}\neq \emptyset$ for any $tt'\in E(T')$. Let $G':=\bigcup_{t\in V(T')}G[\bigcup_{P\in B_t} P]$. Since $T'$ is a subtree of $T$, by a prior observation we made when we first defined the completion, $G'$ is isometric in $G$. Also, notice that for any part $P\in \scr{P}$, $G[P]$ is either a subgraph of $G'$, or $P$ is disjoint from $V(G')$; let $\scr{P}':=\Parts_{\scr{P}}(V(G'))$, we thus have that $\scr{P}'$ is an $\ell$-shallow partition for $G'$. Finally, notice that $Z\subseteq V(G')$.
            
            Next, let $E'$ denote the set of edges between $T'$ and $T-V(T')$. For each $e\in E'$, let $T_e$ be the connected component of $T-V(T')$ incident to $e$; note that $e$ is the only edge between $T_e$ and $T-V(T_e)$. Thus, $T_e$ is disjoint from $T_{e'}$, $e'\in E'\setminus \{e\}$, and there is a unique vertex $q_e\in V(T_e)$ adjacent to $T-V(T_e)$; in particular $q_e$ is adjacent to $T'$ and $e$ is incident with $q_e$. Additionally, notice that $V(T')\cup \bigcup_{e\in E'}V(T_e)=V(T)$; in particular, for each $e\in E'$, $V(T)\setminus V(T_e)=V(T')\cup \bigcup_{e'\in E'\setminus \{e\}}V(T_{e'})$. Finally, notice that because $V(T_e)\subseteq V(T)\setminus V(T')$, for each $t\in V(T_e)$, $B_t\cap Z^\scr{P}=\emptyset$.
        
            Now, for each $e\in E'$, let $G_e:=\bigcup_{t\in V(T_e)}G[\bigcup_{P\in B_t}P]$, and let $\widetilde{G}_e:=\bigcup_{t\in V(T)\setminus V(T_e)}G[\bigcup_{P\in B_t}P]$. Since the only edge between $T_e$ and $T-V(T_e)$ was $e$, whose endpoints are $q_e$ and some vertex of $T'$, note that $V(G_e\cap \widetilde{G}_e)=S_e\subseteq V(G')$ and that $S_e^{\scr{P}}\subseteq B_{q_e}$. Additionally, using the same reasoning as with $G'$, notice that $G_e$ is an isometric subgraph of $G$, and that $\scr{P}_e:=\Parts_{\scr{P}}(V(G_e))$ is an $\ell$-shallow partition of $G_e$. Finally, notice that $Z\cap V(G_e)=\emptyset$, as $Z^\scr{P}\cap B_t=\emptyset$ for each $t\in V(T_e)\subseteq V(T)\setminus V(T')$.

            We now seek to apply \cref{boundComponentsDiameter} on $G$, using $G'$ as $G_0$ and $(G_e:e\in E')$ as $G_1,\dots,G_a$. We use $r$ as both $r$ and $\ell$, $n'+1$ for $m$, $\{1,\dots,n'+1\}$ for $C$, $k$ as itself, $f_k^{\#}(r)$ for $d$, and $f_k(r)$ for $D$. Notice that $\ell'$ is $g_k'(r)$ and $r'$ is $g_k^*(r)$, hence $D=f_k(r)\geq d+2r'=f_k^{\#}(r)+2g_k^*(r)$, as required.

            To begin, notice that each vertex or edge of $G$ is contained in $G[\bigcup_{P\in B_t}P]$ for some $t\in V(T)$. Since $V(T')\cup \bigcup_{e\in E'} V(T_e) = V(T)$, for each $t\in V(T)$, $G[\bigcup_{P\in B_t}P]$ is a subgraph of either $G'$, or $G_e$ for some $e\in E'$. Thus, $G=G'\cup\bigcup_{e\in E'}G_e$, and we have satisfied \cref{boundComponentsDiameter}~(a).

            Next, for each $e\in E'$, recall that $S_e=V(G_e\cap \widetilde{G}_e)\subseteq V(G')$ and that $S_e$ is a $(k,r)$-centred set in $G$. Since $V(T)\setminus V(T_e)=V(T')\cup \bigcup_{e'\in E'\setminus \{e\}}V(T_{e'})$, we have $\widetilde{G}_e=G'\cup\bigcup_{e'\in E'\setminus \{e\}}G_{e'}$. Thus, the $\widetilde{G}_e$ [resp.\ the $S_e$], $e\in E'$, are the $\widetilde{G}_i$ [resp.\ the $S_i$], $i\in \{1,\dots,a\}$, in the statement of \cref{boundComponentsDiameter}, so we have satisfied \cref{boundComponentsDiameter}~(b).

            Now, consider \cref{boundComponentsDiameter}~(c). Notice that $\scr{P}'\setminus Z^{\scr{P}}$ is an $\ell$-shallow partition of $G'-Z$, and that $(B_t\setminus Z^{\scr{P}}:t\in V(T'))$ is tree-decomposition for $(G'-Z)/\scr{P}'$; since $B_t\cap B_{t'}\cap Z^{\scr{P}}\neq \emptyset$ for each $tt'\in E(T')$, we find that $(B_t\setminus Z^{\scr{P}}:t\in V(T'))$ has adhesion at most $k-1$. Also, notice that $(\bigcup_{P\in B_t\setminus Z^{\scr{P}}}P:t\in V(T'))$ is a tree-decomposition of $G'-Z$; fix $t\in V(T')$ and consider $\WT{(G'-Z)}{\bigcup_{P\in B_t \setminus Z^{\scr{P}}}P}$. For any $t'\in V(T')$ adjacent to $t$ and any $u,v\in \bigcup_{P\in (B_t\cap B_{t'})\setminus Z^{\scr{P}}}P$, observe that $uv\in E(G)$ and hence $uv\in E(G'-Z)$. Additionally, notice that for each $uv\in E(G'-Z)$, $uv$ has weight $\dist_G(u,v)$ in both $G$ and $G'-Z$; this forces $\dist_G(u,v)=\dist_{G'-Z}(u,v)$. It follows that $\WT{(G'-Z)}{\bigcup_{P\in B_t \setminus Z^{\scr{P}}}P}=(G'-Z)[\bigcup_{P\in B_t\setminus Z^{\scr{P}}}P]=G[\bigcup_{P\in B_t\setminus Z^{\scr{P}}}P]$. Thus, for any $\scr{P}_t\subseteq B_t\setminus Z^{\scr{P}}\subseteq B_t$, $\WT{(G'-Z)}{\bigcup_{P\in B_t \setminus Z^{\scr{P}}}P}[\bigcup_{P\in \scr{P}_t}P]=G[\bigcup_{P\in \scr{P}_t}P]\in \scr{H}$. Therefore, $(\scr{P}'\setminus Z^{\scr{P}},(B_t\setminus Z^{\scr{P}}:t\in V(T'))$ is a $(k-1,\ell,\scr{H})$-strong-construction for $G'-Z$.
            
            Thus, using $g_k^*(r)\geq \ell$ as $r$ and recalling that $f_{k-1}(g_k^*(r))=f_k^*(r)$, we may apply the induction hypothesis, keeping $q$ unchanged and letting $S^{\scr{P}}$ be empty, to find an $(n'+1,g_k^*(r),f_k^*(r))$-colouring $c'$ of $G'-Z$ with colours $\{1,\dots,n'+1\}$. Now, recall that $\scr{P}'$ is an $\ell$-shallow partition for $G'$; thus, for each $P\in \scr{P}'$, $\wdiam_{G'}(P)\leq 2\ell\leq 2r$. Also, as $N_G^{3r}(S)\subseteq Z\subseteq V(G')$ notice that every $v\in N_G^{3r}(S)$ is still a distance at most $3r$ from $S$ in $G'$. Let $S^*$ be a set containing, for each $P\in S^{\scr{P}}$, exactly one central vertex for $G'[P]$. Then for any $z\in Z$, $\dist_{G'}(z,N^{3r}_G(S))\leq 2r$, $\dist_{G'}(z,S)\leq 5r$, and $\dist_{G'}(z,S^*)\leq 6r$. Since $|S^{\scr{P}}|\leq k$, $|S^*|\leq k$, and thus $Z$ is $(k,6r)$-centred set in $G'$. Since $(k+1)(f_k^*(r)+4g_k^*(r)+12r)=f_k^{\#}(r)$, we thus have that the colouring $c'':=c'\cup c_Z$ of $G'$ is an $(n'+1,g_k^*(r),f_k^{\#}(r))$-colouring with colours $\{1,\dots,n'+1\}$ via \cref{colouringMinusCentredSet}. So \cref{boundComponentsDiameter}~(c) is satisfied with $c''$ as $c_0$. Finally, note that $c''$ also extends $c_Z$, by definition.
            
            Now, for each $e\in E'$, define $c_{S_e}:=c''\big|_{V(S_e)}$. Let $Z_e^{\scr{P}}:=\Parts_{\scr{P}}(N_{G_e}^{3r}(S_e))$ and $Z_e:=\bigcup_{P\in Z_e^{\scr{P}}} P$, we can find a colouring $c_{Z_e}$ of $Z_e$ with colours $\{1,\dots,n'+1\}$ that extends $c_{S_e}$ and has an $(S_e,r)$-barrier in $G_e$. Recall that $\scr{P}_e$ is an $\ell$-shallow partition of $G_e$, and notice that $(B_t:t\in V(T_e))$ is a tree-decomposition for $G_e/\scr{P}_e$ of adhesion at most $k$, and that $(\bigcup_{P\in B_t}P:t\in V(T_e))$ is a tree-decomposition for $G_e$. For any $t\in V(T_e)$, a similar argument to the one used for the weighted torsos of $G'-Z$ shows that $\WT{G_e}{\bigcup_{P\in B_t}P}=G_e[\bigcup_{P\in B_t}P]=G[\bigcup_{P\in B_t}P]$; thus, for any $\scr{P}_t\subseteq B_t$, $\WT{G_e}{\bigcup_{P\in B_t}P}[\bigcup_{P\in \scr{P}_t}P]=G[\bigcup_{P\in \scr{P}_t}P]\in \scr{H}$. It follows that $(\scr{P}_e, (B_t:t\in V(T_e)))$ is a $(k,\ell,\scr{H})$-strong-construction for $G_e$. Finally, note that $|V(G_e)|\leq |V(G)-Z|<|V(G)|$, as $Z\cap V(G_e)=\emptyset$ and $Z$ is nonempty. Hence, we can apply induction on $G_e$, using $r$ as itself, $q_e$ as $q$, $S_e^{\scr{P}}\subseteq B_{q_e}$ as $S^{\scr{P}}$, and $c_{Z_e}$ as $c_Z$. This allows us to extend $c_{Z_e}$ to an $(n'+1,r,f_k(r))$-colouring $c_e$ of $G_e$ with colours $\{1,\dots,n'+1\}$. Note that since $c_e$ extends $c_{Z_e}$, $c_e$ has an $(S_e,r)$-barrier in $G_e$, and that $c_e=c''$ on $S_i$, by definition of $c_{S_e}$. Hence, \cref{boundComponentsDiameter}~(d) is satisfied with the various $c_e$, $e\in E'$, acting as $c_1,\dots,c_a$.

            Thus, we can now apply \cref{boundComponentsDiameter} with the parameters specified earlier. This gives us that $c:=c''\cup\bigcup_{e\in E'}c_e$ is an $(n'+1,r,f_k(r))$-colouring with colours $\{1,\dots,n'+1\}$. Since $c$ extends $c''$, which extends $c_Z$, $c$ also extends $c_Z$. Thus, $c$ is the desired colouring.
    \end{proof}

    \section{Colourings on Proper Minor-Closed Classes}

    A landmark result by \citet*{Robertson2003} known as the ``Graph Minor Structure Theorem" gives a structural description of $H$-minor free graphs. We now build towards stating this result.
    
    The \defn{Euler genus} of a surface with $h$ handles and $c$ cross-caps is $2h+c$. The \defn{Euler genus} of a graph $G$ is the minimum Euler genus of a surface $\Sigma$ such that $G$ can be embedded into $\Sigma$ without crossings.

    Given a graph $G_0$ embedded in some surface $\Sigma$ without crossings, a closed disc $D$ in $\Sigma$ is said to be \defn{$G_0$-clean} if the interior of $D$ does not intersect the embedding of $G_0$, and the boundary of $D$ intersects the embedding of $G_0$ only at vertices of $G_0$. Observe that $D$ must sit inside a face of $G_0$; if the vertices of $G_0$ that $D$ intersects are precisely the vertices of this face, we say that $D$ is \defn{snug} in $G_0$. There is a natural cyclic ordering of the vertices of $G_0$ that $D$ intersects by following the boundary of $D$; let $v_1, \dots, v_n$ denote these vertices in this order. Observe that if $D$ is snug, this ordering is the same as the cyclic ordering of the corresponding face. A \defn{$D$-vortex} of $G_0$ is then a graph $H$ such that $V(G_0\cap H)=\{v_1,\dots,v_n\}$. If $P_n$ denotes the path on $n$ vertices, then a \defn{vortex-decomposition} of a $D$-vortex $H$ of $G_0$ is a $P_n$-decomposition $(B_1,\dots, B_n)$ of $H$ such that $v_i\in B_i$ for each $i\in \{1,\dots,n\}$. The \defn{width} of a vortex-decomposition is the width of the $P_n$-decomposition, and the \defn{width} of a $D$-vortex $H$ of $G_0$ is the minimum width of a vortex-decomposition of $H$. Note that the treewidth of $H$ is at most the width of $H$.

    For integers $g,p,k,a\geq 0$, a graph $G$ is said to be \defn{$(g,p,k,a)$-almost embeddable} if there exists a set $A\subseteq V(G)$ of size at most $a$ and subgraphs $G_0,\dots,G_s$ of $G$ for some $s\in \{0,\dots,p\}$ such that:

    \begin{enumerate}[label=(\alph*)]
        \item $G-A=\bigcup_{i=0}^s G_i$;
        \item $G_1,\dots,G_s$ are pairwise disjoint;
        \item $G_0$ is embedded into a surface $\Sigma$ of Euler genus at most $g$;
        \item there exist pairwise disjoint $G_0$-clean discs $D_1,\dots,D_s$ in $\Sigma$; and
        \item $G_i$ is a $D_i$-vortex of $G_0$ of width at most $k$ for each $i\in \{1,\dots,s\}$.
    \end{enumerate}

    Additionally, if all the discs $D_1,\dots,D_s$ are snug in $G_0$, we say that $G$ is \defn{snugly $(g,p,k,a)$-almost embeddable}.

    Refer to the set $A$ as the \defn{apex vertices} of $G$, the subgraph $G_0$ as the \defn{embedded subgraph} of $G$, the $G_0$-clean closed discs $D_1, \dots, D_s$ as the \defn{discs} of $G$, and the subgraphs $G_1, \dots, G_s$ as the \defn{vortices} of $G$; for each $i\in \{1,\dots,s\}$, $G_i$ is the \defn{vortex} at $D_i$. Also, call $H:=\bigcup_{i=1}^s G_i$ the \defn{vortex-union subgraph} of $G$. Next, for $i\in \{1,\dots,s\}$, refer to the vertices $S_i$ of $G_0\cap G_i$, ordered by the natural cyclic ordering from the boundary of $D_i$, as the \defn{boundary vertices} of $D_i$. The set $S:=\bigcup_{i=1}^s S_i$ is referred to as the \defn{boundary} of $G$.
    
    We take the chance to note some properties of snugly $(g,p,k,a)$-almost embeddable graphs. For each $i\in \{1,\dots,s\}$, let $(B_{i,1},\dots,B_{i,n_i})$ be a vortex-decomposition of $G_i$ of width at most $k$, and let $S_i=v_{i,1},\dots,v_{i,n_i}$ be the cyclic ordering of the boundary vertices of $D_i$; as $D_i$ is snug $v_{i,1},\dots,v_{i,n_i}$ is precisely the vertex set of the face $F_i$ that $D_i$ sits inside, in the cyclic order induced by $F_i$. Thus, for any interval $a,\dots,b$ of the path $P_{n_i}$, $v_{i,a},\dots,v_{i,b}$ induces a connected subgraph of $G_0[S_i]$. It follows that $(B_{i,j}:v_{i,j}\in S_i)$ is a rooted $G_0[S_i]$-decomposition of $G_i$ of width at most $k$. As the vortices (including the boundary vertices) as pairwise disjoint, we find that $(B_{i,j}:v_{i,j}\in S)$ is a rooted $G_0[S]$-decomposition of $H$ of width at most $k$. Additionally, since for each $i\in \{1,\dots,s\}$, $S_i$ is the vertex set of $F_i$, $S$ is the vertex set of the union of the faces $F_1,\dots,F_s$. Finally, observe that $(G_0,H)$ forms a separation of $G-A$ with separator $V(G_0\cap H)=S$.
    
    We can now state the Graph Minor Structure Theorem of \citet*{Robertson2003}.

    \begin{theorem}[Graph Minor Structure Theorem~\citep{Robertson2003}]
        \label{GMST}
        For every graph $H$, there exists an integer $k\geq 0$ such that every $H$-minor-free graph $G$ admits a tree-decomposition of adhesion at most $k$ such that every torso is $(k,k,k,k)$-almost embeddable.
    \end{theorem}

    Observe that if $G$ as above is a weighted graph, every weighted torso is also $(k,k,k,k)$-almost embeddable. Thus, if $\scr{H}$ is the class of all $(k,k,k,k)$-almost embeddable weighted graphs, $G$ is $(k,\scr{H})$-constructable. However, since induced subgraphs of $(k,k,k,k)$-almost embeddable graphs are not necessarily $(k,k,k,k)$-almost embeddable, $\scr{H}$ is not hereditary, and we consequently cannot make use of \cref{ANdimConstructable}. Instead, consider the class $\scr{H}'$ of all induced subgraphs of $(k,k,k,k)$-almost embeddable weighted graphs; note that $\scr{H}'$ is hereditary and $G$ is also $(k,\scr{H}')$-constructable. Thus, if we could bound the Assouad--Nagata dimension of $\scr{H}'$, we could apply \cref{ANdimConstructable} to get a bound on the Assouad--Nagata dimension of any $H$-minor-free class. Since proper minor-closed classes are $H$-minor-free for some graph $H$, to prove \cref{ANdimMCC}~(a) it suffices to show that $\ANdim(\scr{H}')\leq 2$. 
    
    We now focus on finding a dilation $f$ that is a $2$-dimensional control function for every weighted graph that is an induced subgraph of a $(g,p,k,a)$-almost embeddable weighted graph. We show that every such weighted graph is isometric in a snugly $(g,p,k,a)$-almost embeddable weighted graph; see \cref{snugSupergraph}. Thus, it suffices to show that $f$ is a $2$-dimensional control function for every snugly $(g,p,k,a)$-almost embeddable weighted graph $G$. We notice that if we can find a control function for $G-A=G_0\cup H$, we can then apply \cref{colouringMinusCentredSet} to find $f$ as $A$ is a $(a,0)$-centred set. So we just need to focus on the union $G_0\cup H$. Individually, $G_0$ and $H$ admit $2$-dimensional control functions, as the former has bounded Euler genus, so \cref{ANdimGenus} applies, and the latter has bounded treewidth, so \cref{ANdimTw} applies. However, taking the union creates issues, as it no longer has bounded treewidth nor Euler genus. Indeed, we have no direct way of dealing with this union; trying to colour the two halves of the separation independently fails, as by jumping back and forth between $G_0$ and $H$ we can find short paths in the union that exist in neither half individually. However, we can recognise that the problem occurs ``close" to where the two subgraphs meet, the boundary $S$ of $G$, and that ``far away" from the boundary, naively colouring as either a graph of bounded Euler genus or as a graph of bounded treewidth suffices. So we just need to deal with the part of the graph ``close" to the boundary. We formalise this idea in the following proposition.

    \begin{proposition}
        \label{colouringSeparation}
        Let $n\geq 0$ be an integer, let $f,g:\ds{R}^+\rightarrow \ds{R}^+$, let $G$ be a weighted graph, and let $(A,B)$ be a separation of $G$ with separator $S$ such that:

        \begin{enumerate}[label=(\alph*)]
            \item $f$ is an $n$-dimensional control function for both $A$ and $B$; and
            \item for any real number $\ell\geq 0$, $g$ is an $\ell$-almost $n$-dimensional control function for $G[N_G^\ell(S)]$.
        \end{enumerate}
        Then the function $r\mapsto g(f(r)+4r)+2(f(r)+4r)$ is an $n$-dimensional control function for $G$.
    \end{proposition}

    \begin{proof}
        Fix a real number $r>0$, set $r':=f(r)+4r$ and $d':=g(r')$. We need to show that $G$ admits an $(n+1,r,d'+2r')$-colouring. Let $c_A$ and $c_B$ be $(n+1,r,f(r))$-colourings of $A,B$ respectively, set $Z:=N_G^r(S)$ and $Z':=N_G^{r'}(S)$, and let $c_{Z'}$ be an $(n+1,r',d')$-colouring of $G[Z']$. Let $\iota: Z\rightarrow S$ be the identity on $S$, and for $x\in Z\setminus S$, let $\iota(x)$ be a vertex in $S\cap N_G^r(x)$. Now, define a colouring $c$ of $G$ via, for each $x\in V(G)$:
        \[c(x):= \begin{cases} 
            c_A(x) & \text{if }x\in V(A)\setminus Z, \\
            c_B(x) & \text{if }x\in V(B)\setminus Z, \\
            c_{Z'}(\iota(x)) & \text{otherwise.}
            \end{cases}
        \]
        Observe that for any $x,y\in V(A)\setminus Z$ at distance at most $r$ in $G$ and any shortest path $Q$ from $x$ to $y$ in $G$, $Q$ cannot intersect $S$ as otherwise we would have $x,y\in N_G^r(S)=Z$. Since $Q$ induces a connected subgraph of $G$, this implies that $Q\subseteq V(A)-S$. Noting that every edge of $G$ between vertices in $V(A)-S$ is an edge of $A$, $Q$ is therefore also a path of length at most $r$ in $A$, and $\dist_A(x,y)\leq r$. Consequently, any $r$-path in $G$ whose vertices are contained in $V(A)\setminus Z$ is also an $r$-path in $A$. A symmetric argument shows that any $r$-path in $G$ whose vertices are contained in $V(B)\setminus Z$ is also an $r$-path in $B$.
        
        Now, notice that for any $x\in V(A)$ and $y\in V(B)$, if $\dist_G(x,y)\leq r$, then $x,y\in N_G^r(S)=Z$, as the shortest path between them must intersect $S$. Consider an $r$-path $P'$ in $G$ that is disjoint from $Z$; we claim that $P'$ must be contained in either $V(A)\setminus Z$ or $V(B)\setminus Z$. Otherwise, somewhere on $P'$, we would have consecutive vertices $x\in V(A)$, $y\in V(B)$; the previous observation then tells us that $x,y\in Z$, a contradiction. Therefore, for any monochromatic $r$-path $P$ from $u\in V(G)$ to $v\in V(G)$ in $G$ under $c$, any $r$-subpath $P'$ of $P-Z$ must be contained in either $V(A)\setminus Z$, or in $V(B)\setminus Z$. The observation from the previous paragraph then tells us that $P'$ is an $r$-path in either $A$, or in $B$, respectively. Since $c$ is $c_A$ on $V(A)\setminus Z$ [resp.\ $c_B$ on $V(B)\setminus Z$], $P'$ is also monochromatic in either $A$ under $c_A$, or in $B$ under $c_B$, respectively. Thus, $P'$ has weak diameter at most $f(r)$ in either $A$, or in $B$, respectively, and consequently in $G$ as well. 
        
        Therefore, we can apply \cref{rerouting} to find that there exists a $f(r)+2r+2r=r'$-walk $P''$ from $u$ to $v$ in $G$ whose interior consists of vertices of the form $\iota(x)$ with $x\in P\cap Z$. Observe that for $x\in P\cap Z$, $c(x)=c_{Z'}(\iota(x))=c_{Z'}(\iota(\iota(x)))=c(\iota(x))$, as $\iota(\iota(x))=\iota(x)$ since $\iota$ is the identity on $S$. Therefore, $P''$ is also monochromatic.
        
        Now, for any $x,y\in S$ at distance at most $r'$ in $G$ and any shortest path $Q$ from $x$ to $y$ in $G$, observe that $Q$ is contained in $N_G^{r'}(S)=Z'$. Thus, $Q$ is also a path of length at most $r'$ in $G[Z']$, and $\dist_{G[Z']}(x,y)\leq r'$. Since $\Int(P'')$ is contained in $S$, $\Int(P'')$ is also a (possibly empty) $r'$-path in $G[Z']$. Further, since $c$ is $c_{Z'}$ on $S$, as $\iota$ is the identity map on $S$, we also have that $\Int(P'')$ is monochromatic under $c_{Z'}$. Hence, $\Int(P'')$ has weak diameter at most $d'$ in both $G[Z']$ and $G$. Thus, $P''$ has weak diameter at most $d'+2r'$ in $G$; in particular, $\dist_G(u,v)\leq d'+2r'$, as desired.
    \end{proof}

    So we just need to find a dilation that is an $\ell$-almost $2$-dimensional control function for $(G-A)[N_{G-A}^\ell(S)]$. In fact, we find $\ell$-almost $1$-dimensional control function for $(G-A)[N_{G-A}^\ell(S)]$. We do this by showing $(G-A)[N_{G-A}^\ell(S)]$ admits a $(w,\ell)$-partition, across multiple steps; \cref{ANdimPartition} then gives the desired dilation. The first step is to find a ``natural" $\ell$-shallow partition of $G_0[N_{G_0}^\ell(S)]$.

    \begin{proposition}
        \label{partitionNeighbourhood}
        Let $r\geq 0$ be a real number, let $G$ be a weighted graph, and let $S\subseteq V(G)$. Then $G':=G[N_G^r(S)]$ admits an $r$-shallow partition $\scr{P}$ such that each part of $\scr{P}$ contains exactly one vertex of $S$.
    \end{proposition}

    \begin{proof}
        Let $\preceq$ be an arbitrary ordering of the vertices of $S$. For $v\in N_G^r(S)$, let $\iota(v)$ be the smallest, with respect to $\preceq$, $s\in S$ such that $\dist_G(v,S)=\dist_G(v,s)$. Then, for each $s\in S$, define $P_s:=\{v\in N_G^r(S): \iota(v)=s\}$, and finally, define $\scr{P}:=\{P_s:s\in S\}$. We argue that $\scr{P}$ is the desired partition of $G'$.

        Each vertex of $G'$ is in exactly one part of $\scr{P}$, as $\preceq$ acts as a tiebreaker. Additionally, for each $s\in S$, $s\in P_s$ as $\dist_G(s,s)=0$ and $\dist_G(s,s')>0$ for $s'\in S\setminus \{s\}$; it follows that each part is nonempty and contains exactly one vertex of $S$. So it remains only to show that each part $P_s$, $s\in S$, induces a connected subgraph of radius at most $r$; it suffices to show that for each $v\in P_s$, there exists a path from $v$ to $s$ in $G'[P_s]$ of length at most $r$. 
        
        Let $Q$ be a shortest path from $v$ to $s$ in $G$; by definition of $P_s$, $Q$ must have length $\dist_G(v,S)\leq r$. Thus, $Q\subseteq N_G^r(S)=V(G')$, as each vertex of $Q$ is at least as close to $s$ as $v$ is. Consequently, each vertex of $Q$ lies in some part of $\scr{P}$; assume, for a contradiction, that some $u\in Q$ is not in $P_s$, and instead lies in $P_{s'}$ for some $s'\in S\setminus \{s\}$. Since $Q$ is a shortest path, $\dist_G(s,v)=\dist_G(s,u)+\dist_G(u,v)$, and since $u\in P_{s'}$, $v\in P_s$, $\dist_G(s,u)\geq \dist_G(s',u)$ and $\dist_G(s',v)\geq \dist_G(s,v)$. If $\dist_G(s,u)>\dist_G(s',u)$, then $\dist_G(s,v)>\dist_G(s',u)+\dist_G(u,v)\geq \dist_G(s',v)$, a contradiction. So we must have that $\dist_G(s,u)=\dist_G(s',u)$ and $\dist_G(s,v)=\dist_G(s',u)+\dist_G(u,v)\geq \dist_G(s',v)$; this forces $\dist_G(s,v)=\dist_G(s',v)$. If $s'\prec s$, since $\dist_G(s',v)=\dist_G(s,v)=\dist_G(v,S)$, $\preceq$ would put $v$ in $P_{s'}$ over $P_s$, contradicting the fact that $v\in P_s$. By contrast, if $s'\succ s$, since $\dist_G(s,u)=\dist_G(s',u)=\dist_G(u,S)$, $u$ would have been placed in $P_s$ over $P_{s'}$, another contradiction. Thus, we must conclude that our assumption was false. This gives $Q\subseteq P_s$, so $Q$ is also a path from $v$ to $s$ in $G'[P_s]$, and is of the same length in $G'[P_s]$ as in $G$. Since $Q$ had length at most $r$ in $G$, $Q$ is therefore the desired path.
    \end{proof}

We need the following lemma from \citet*{Dujmovic2017}; see \citep{Eppstein2000} for an earlier $O(gr)$ bound.
    
    \begin{lemma}
        \label{twGenusRadius}
        For all integers $g,r\geq 0$, every unweighted graph of radius at most $r$ with Euler genus at most $g$ has treewidth at most $(2g+3)r$.
    \end{lemma}

    We now show that the partition of $G_0[N_{G_0}^\ell(S)]$ we obtain from \cref{partitionNeighbourhood} is actually a $(w,\ell)$-partition.

    \begin{proposition}
        \label{twFacialPartition}
        Let $g\geq 0$, $p\geq 0$ be integers, let $r\geq 0$ be a real number, and let $G$ be a weighted graph embedded in a surface $\Sigma$ of Euler genus at most $g$ without crossings. Let $F_1,\dots,F_s$, $s\leq p$, be faces of $G$, let $S$ be the vertex set of the union of $F_1,\dots,F_s$, set $G':=G[N_G^r(S)]$, and let $\scr{P}$ be an $r$-shallow partition of $G'$ such that each part contains exactly one vertex in $S$. Then $\scr{P}$ is a $(2g+4p+3,r)$-partition of $G'$.
    \end{proposition}

    \begin{proof}
        We must show that $G/\scr{P}$ has treewidth at most $2g+4p+3$. For each $i\in \{2,\dots,s\}$, add a handle connecting the interior of $F_1$ to the interior of $F_i$; this gives an embedding of $G'$ into a surface $\Sigma'$ of Euler genus at most $g+2(\max(s-1,0))\leq g+2p$ such that all the vertices in $S$ are in a common face. Thus, by viewing $G'/\scr{P}$ as contracting each part down to the unique vertex in $S$, we can see that $G'/\scr{P}$ can be embedded in $\Sigma'$ so that every vertex lies on a common face. We can then add a new vertex to $G'/\scr{P}$ adjacent to every existing vertex within the interior of the common face to create a new unweighted graph $H$ of radius at most 1 embedded in $\Sigma'$. Thus, by \cref{twGenusRadius}, $H$ has treewidth at most $2(g+2p)+3=2g+4p+3$. Since $G'/\scr{P}\subseteq H$, $G'/\scr{P}$ also has treewidth at most $2g+4p+3$, as desired.
    \end{proof}

    Finally, we extend the $(w,\ell)$-partition of $G_0[N_{G_0}^\ell(S)]$ to a $(w',\ell)$-partition of $(G-A)[N_{G-A}^\ell(S)]$.

    \begin{proposition}
        \label{twSeparation}
        Let $t,w\geq 0$ be integers, let $r\geq 0$ be a real number, and let $G'$ be a graph that admits a separation $(G,H)$ with separator $S$ such that:
            \begin{enumerate}[label=(\alph*)]
                \item $G$ admits a $(t,r)$-partition $\scr{P}$ such that each part contains exactly one vertex in $S$; and
                \item $H$ admits a rooted $G[S]$-decomposition of width at most $w$.
            \end{enumerate}
        Then $\scr{P}':=\scr{P}\cup\bigcup_{v\in V(H)\setminus S}\{\{v\}\}$ is a $((t+1)(w+1)-1,r)$-partition of $G'$.
    \end{proposition}

    \begin{proof}
        It is immediate that $\scr{P}'$ is an $r$-shallow partition of $G'$, so we only need to show that the treewidth of $G'/\scr{P}'$ is at most $(t+1)(w+1)-1$.

        Let $(B_t:t\in V(T))$ be a tree-decomposition of $G/\scr{P}$ of width at most $t$, and let $(J_s:s\in S)$ be a rooted $G[S]$-decomposition of $H$ of width at most $w$. For each $t\in V(T)$, let $S_t:=\bigcup_{P\in B_t}P\cap S$; observe that $|S_t|\leq t+1$ as $(B_t:t\in V(T))$ has width at most $t$ and each $P\in B_t$ contains only one vertex in $S$. For each $t\in V(T)$, let $K_t:=\bigcup_{s\in S_t}\Parts_{\scr{P}'}(J_s)$. Note that for each $P\in B_t$, there exists $s\in P\cap S_t$; since $s\in J_s$, this gives $P=\Part_{\scr{P}'}(s)\in K_t$. Thus, $B_t\subseteq K_t$. We now argue that $(K_t:t\in V(T))$ is the desired tree-decomposition of $G'/\scr{P}'$.

        First, we show that $(K_t:t\in V(T))$ is a tree-decomposition of $G'/\scr{P}'$. Consider any $h\in V(H)$, the vertices $s\in S$ such that $h\in J_s$ induce an nonempty connected subgraph $C_h$ of $G[S]$. Let $C_h^\scr{P}=G/\scr{P}[\Parts_{\scr{P}}(V(C_h))]$ note that $C_h^\scr{P}$ must also be nonempty and connected. Thus, the vertices $t\in V(T)$ such that $B_t\cap V(C_h^\scr{P})\neq \emptyset$ induce a nonempty subtree $T_h$ of $T$. Also, observe that for each $s\in S$, $\Part_{\scr{P}}(s)\in V(C_s^\scr{P})$, as $(J_s:s\in S)$ is rooted. Now, for each $P\in \scr{P}'$, let $T'_P$ be the subgraph of $T$ induced by the vertices $t\in V(T)$ such that $P\in K_t$; we need to show that $T'_P$ is nonempty and connected. Note that there is exactly one vertex $h_P$ in $P\cap V(H)$; either the unique $s$ in $P\cap S$ if $P\in \scr{P}$, or the $h$ such that $P=\{h_P\}$ for $P\in \scr{P}'\setminus\scr{P}$. We argue that $T'_P=T_{h_P}$; since $T_{h_P}$ is nonempty and connected, this gives the desired result. Note that it suffices to show that $V(T'_P)=V(T_{h_P})$, as both subgraphs are induced.

        For a given $P\in \scr{P}'$, if $t\in V(T)$ is in $V(T'_P)$, then $P\in \Parts_{\scr{P}'}(J_s)$ for some $s\in S_t$. However, since $P\cap V(H)=\{h_P\}$ and $J_s\subseteq V(H)$, this forces $h_P\in J_s$. This gives $s\in V(C_{h_P})$, and $\Part_{\scr{P}}(s)\in V(C_{h_P}^\scr{P})$. By definition of $S_t$, $\Part_{\scr{P}}(s)\in B_t$, and thus $V(C_{h_P}^\scr{P})\cap B_t\neq \emptyset$. Therefore, $t\in V(T_{h_P})$. By contrast, for each $t\in V(T_{h_P})$, there exists $P'\in B_t\cap V(C_{h_P}^\scr{P})$; this part $P'$ must contain a $s\in V(C_{h_P})\subseteq S$. By definition of $C_{h_P}$, we have ${h_P}\in J_s$, so $P=\Part_{\scr{P}'}(h_P)\in \Parts_{\scr{P}'}(J_s)$. However, we also know that $s\in P'\cap S\subseteq S_t$, as $P'\in B_t$, so $P\in \bigcup_{s\in S_t}\Parts_{\scr{P}'}(J_s)=K_t$. Therefore, $t\in V(T'_P)$; this combined with the other direction above gives $V(T'_P)=V(T_h)$, as desired.

        To complete our proof that $(K_t:t\in V(T))$ is a tree-decomposition of $G'/\scr{P}'$, it remains only to show that for any $P_1P_2\in E(G'/\scr{P}')$, $P_1,P_2\in K_t$ for some $t\in V(T)$. We know that there is some $uv\in E(G')$ such that $u\in P_1$, $v\in P_2$, which must be an edge in either $G$ or $H$. In the former case, $P_1$, $P_2$ are both parts in $\scr{P}$ and $P_1P_2$ is also an edge of $E(G/\scr{P})$; thus, there is some $t\in V(T)$ for which $P_1,P_2\in B_t\subseteq K_t$. Otherwise, there exists $s\in S$ with $u,v\in J_s$, and thus $P_1=\Part_{\scr{P}'}(u),P_2=\Part_{\scr{P}'}(v)\in K_t$ for some $t\in V(T)$ for which $\Part_{\scr{P}}(s)\in B_t$.

        Lastly, we just need to consider the width of our tree-decomposition. Observe that for each $t\in V(T)$, $|K_t|\leq \sum_{s\in S_t}|J_s|\leq |S_t|(w+1)\leq (t+1)(w+1)$, as desired.
    \end{proof}

    All that is left is to prove that any induced subgraph of a $(g,p,k,a)$-almost embeddable weighted graph is isometric in a snugly $(g,p,k,a)$-almost-embeddable weighted graph.

    %\marc{Has issues if the disc intersects <= 2 vertices}
    \begin{proposition}
        \label{snugSupergraph}
        Let $g,p,a,k\geq 0$ be integers, and let $H$ be an induced subgraph of a $(g,p,k,a)$-almost embeddable weighted graph. Then $H$ is an isometric subgraph of a snugly $(g,p,k,a)$-almost embeddable weighted graph.
    \end{proposition}

    \begin{proof}

        We show that $H$ is an isometric subgraph of a $(g,p,k,a)$-almost embeddable weighted graph $G$, and that $G$ is an isometric subgraph of a snugly $(g,p,k,a)$-almost embeddable weighted graph $G'$. The result follows since being isometric is a transitive property. We start with the former statement. 
        
        By definition of $H$, $H$ is an induced subgraph of a $(g,p,k,a)$-almost embeddable weighted graph $G^{\#}$. Let $w_H$ be the weighting function of $H$, and let $w_G:E(G)\rightarrow \ds{R}^+\cup \{\infty\}$ be defined by setting $w_G(e):=w_H(e)$ for each $e\in E(H)$ and setting $w_G(e):=\infty$ otherwise. Let $G$ be the result of imbuing $\UNW(G^{\#})$ with $w_G$; observe that $G$ is still $(g,p,k,a)$-almost embeddable, and that $H\subseteq G$ as $w_H=w_G\big|_{E(H)}$. It remains to show that $H$ is isometric in $G$. 
        
        Observe that every path in $G$ is either a path in $H$, in which case it has the same length in $G$ as in $H$ as all the edges have the same weight in $G$ as $H$, or has an edge in $E(G)\setminus E(H)$ and hence has length $\infty$. Since $H\subseteq G$, it follows that $H$ is isometric in $G$, as desired. So we now tackle the latter statement.

        %\marc{Sigma unused?} %\marc{note that $n_i=0$ is redundant, as then $V(G_i)$ is empty}
        Let $G_0$ be the embedded subgraph of $G$, which is embedded in some surface $\Sigma$ of Euler genus at most $g$, and let $D_1,\dots,D_s$ be the discs of $G$. For each $i\in \{1,\dots,s\}$, let $V_i:=\{v_{i,1},\dots,v_{i,n_i}\}$ represent the boundary vertices of $D_i$, where $v_{i,1},\dots,v_{i,n_i}$ is ordered according to the natural cyclic ordering obtained from following the boundary of $D_i$. Note that we may assume $n_i\geq 3$ for each $i\in \{1,\dots,s\}$; otherwise add isolated vertices $v_{i,n_{i}+1},\dots,v_{i,3}$ to $G$, $G_i$ and $G_0$ along the boundary of $D_i$ between $v_{i,n_i}$ and $v_{i,1}$ (if they exist) in that order, and update the vortex-decomposition of $G_i$ by adding the bags $B_j:=v_j$ for $j\in \{n_{i+1},\dots,3\}$. $G$ is isometric in this updated graph $G^*$, so whatever $G^*$ is isometric in $G$ will be isometric in as well; thus we may proceed by setting $G:=G^*$. 
        
        Create $G'_0$ from $G_0$ by, for each pair $i\in \{1,\dots,s\}$ and $j\in \{1,\dots,n_i\}$, adding the edge $v_{i,j}v_{i,j+1}$ if it doesn't already exist. If $v_{i,j}v_{i,j+1}\in E(G)$, we give $v_{i,j}v_{i,j+1}$ the same weight in $G_0'$ as its weight in $G$, otherwise we give it weight $\infty$. We then create $G'$ from $G$ by adding all of the above edges that didn't already exist in $G$, which will all have weight $\infty$.
        
        Define an embedding of $G_0'$ in some surface of Euler genus at most $g$ as follows. Starting from the embedding of $G_0$, for each pair $i\in \{1,\dots,s\}$ and $j\in \{1,\dots,n_i\}$, we embed $v_{i,j}v_{i,j+1}$ by tracing just outside the boundary of $D_i$; if $v_{i,j}v_{i,j+1}\in E(G_0)$, this overrides the existing embedding of $v_{i,j}v_{i,j+1}$ used for $G_0$. Provided that the edges cling sufficiently tightly to the boundaries of the discs, the resulting embedding has no crossings. These edges also form a cycle around $D_i$ as $n_i\geq 3$. Thus, for each $i\in \{1,\dots,s\}$, $D_i$ is $G_0'$-clean, the vertices of $G_0'$ that $D_i$ intersects is $V_i$, and the cyclic ordering from following the boundary of $D_i$ is just the cyclic ordering of $V_i$. Additionally, $D_i$ is nested inside a face $F_i$ whose vertices are $V_i$ and whose induced cyclic ordering is the cyclic ordering of $V_i$. Thus, $D_i$ is snug, and if $G_i$ is the vortex of $G$ at $D_i$, then $G_i$ is also a $D_i$-vortex of $G_0'$, and any vortex-decomposition of $G_i$ as a $D_i$-vortex of $G_0$ is still a vortex-decomposition of $G_i$ as a $D_i$-vortex of $G_0'$. Thus, the width of $G_i$ as a $D_i$-vortex of $G_0'$ is still at most $k$.
        
        Therefore, if $A$ denotes the set of apex vertices of $G$, then $G'-A=G_0'\bigcup_{i=1}^s G_i$, where $G_1,\dots,G_s$ are pairwise disjoint, $G_0'$ is embedded onto a surface of genus at most $g$, and each $G_i$, $i\in \{1,\dots,s\}$, is a $D_i$-vortex of $G'_0$ of width at most $k+1$, where the $D_i$, $i\in \{1,\dots,s\}$ are pairwise disjoint $G_0'$-clean snug discs. Since $|A|\leq a$, this shows that $G'$ is snugly $(g,p,k,a)$-almost embeddable, as desired.
        
        So it remains only to show that $G$ is isometric in $G'$. However, since $G\subseteq G'$, each $e\in E(G)$ has the same weight in $G'$, and each $e\in E(G')\setminus E(G)=\bigcup_{i=1}^s E_i'$ has weight $\infty$, an argument identical to that we used to show that $H$ was isometric in $G$ can be used to show that $G$ is isometric in $G'$.
    \end{proof}

    We are finally in a position to show that the class of all induced subgraphs of $(g,p,k,a)$-almost embeddable weighted graphs has Assouad--Nagata dimension 2.
    
    \begin{proposition}
        \label{ANdimAlmostEmbeddable}
        Let $g,p,a,k\geq 0$ be integers. Let $\scr{G}$ be the class of all $(g,p,k,a)$-almost embeddable weighted graphs, and let $\scr{H}$ be the class of all induced subgraphs of graphs in $\scr{G}$. Then $\scr{H}$ is hereditary and $\ANdim(\scr{H})\leq 2$. 
    \end{proposition}

    \begin{proof}
        The fact that $\scr{H}$ is hereditary is immediate, so we only need to show the existence of a dilation $f$ that is a $2$-dimensional control function for $\scr{H}$. By \cref{ANdimGenus}, there exists a dilation $f^{\#}_1$ that is a $2$-dimensional control function for the class of weighted graphs of Euler genus at most $g$, and by \cref{ANdimTw}, there exists a dilation $f^{\#}_2$ that is a $1$-dimensional control function for the class of weighted graphs of treewidth at most $k$. Let $f^{\#}:=\max(f^{\#}_1,f^{\#}_2)$; note that $f^{\#}$ must also be a dilation. By \cref{ANdimPartition}, there exists a dilation $f'$ such that for every real number $\ell\geq 0$, $f'$ is an $\ell$-almost $1$-dimensional control function for every weighted graph $G$ that admits a $((2g+4p+4)(k+1)-1,\ell)$-partition. Let $f^*(r):=f'(f^{\#}(r)+4r)+2(f^{\#}(r)+4r)$ for any real number $r>0$; note that $f^*$ is also a dilation as both $f^{\#}$, $f'$ are. Finally, let $f(r):=(a+1)(f^*(r)+4r)$ for any real number $r>0$; note that $f$ is also a dilation. We show that $f$ is the desired $2$-dimensional control function for $\scr{H}$.
        
        Consider any $H^{\#}\in \scr{H}$, using \cref{snugSupergraph} we know that $H^{\#}$ is an isometric subgraph of some $G\in \scr{G}$ that is snugly $(g,p,k,a)$-almost embeddable. Let $A$ be the apex vertices of $G$, let $G_0$ denote the embedded subgraph of $G$, let $H$ denote the vortex-union subgraph of $G$, and let $S$ denote the boundary of $G$. As observed when we discussed some properties of snugly almost embeddable graphs, $S$ is the vertex set of the union of the faces $F_1,\dots,F_s$ for which the discs $D_1,\dots,D_s$, $s\leq p$, of $G$ are snug in, $(G_0,H)$ is a separation of $G-A$ with separator $S$, and $H$ admits a $G_0[S]$-decomposition of width at most $k$. We seek to apply \cref{colouringSeparation} using $G-A$ as $G$, $(G_0,H)$ as $(A,B)$, $f^{\#}$ as $f$ and $f'$ as $g$. \cref{colouringSeparation}~(a) is immediately satisfied by definition of $f^{\#}$, as $G_0$ has Euler genus at most $g$ and $H$ has treewidth at most $k$. Therefore, we only need to show \cref{colouringSeparation}~(b); that for any $\ell\geq 0$, $f'$ is an $\ell$-almost $2$-dimensional control function for $(G-A)[N_{G-A}^\ell(S)]$. We show that $(G-A)[N_{G-A}^\ell(S)]$ admits a $((2g+4p+4)(k+1)-1,\ell)$-partition; the desired result then follows by definition of $f'$.
        
        Let $\ell\geq 0$ be a real number, and set $S^\ell:=N^\ell_{G-A}(S)$, $S^\ell_{G_0}:=S^\ell\cap V(G_0)$ and $S^\ell_H:=S^\ell\cap V(H)$. Note that $S^\ell_{G_0}=N_{G_0}^\ell(S)$, and that $S^\ell_{H}=N_{H}^\ell(S)$; we use these facts implicitly in the remainder of this proof. By \cref{partitionNeighbourhood}, $G_0[S^\ell_{G_0}]$ admits an $\ell$-shallow partition $\scr{P}$ such that each part contains exactly one vertex of $S$. Since $G_0$ is embedded on a surface of Euler genus at most $g$ and $S$ is the vertex set of the union of the faces $F_1,\dots,F_s$, with $s\leq p$, by \cref{twFacialPartition}, $\scr{P}$ is $(2g+4p+3,\ell)$-partition. 
        
        Now, notice that $(G_0[S^\ell_{G_0}],H[S^\ell_H])$ is a separation of $(G-A)[S^\ell]$ with separator $S$, and that $H[S^\ell_H]$ also admits a rooted $G_0[S]=(G_0[S_{G_0}^\ell])[S]$-decomposition of width at most $k$. Thus, by \cref{twSeparation}, $\scr{P}':=\scr{P} \cup\bigcup_{v\in S^\ell_H\setminus S}\{\{v\}\}$ is a $((2g+4p+4)(k+1)-1,\ell)$-partition of $(G-A)[S^\ell]$, as desired.

        Thus, we can apply \cref{colouringSeparation} with the parameters specified above to find that $f'(f^{\#}(r)+4r)+2(f^{\#}(r)+4r)=f^*(r)$ is a $2$-dimensional control function for $G-A$. Observing that $A$ is a $(a,0)$-centred set in $G$, by \cref{colouringMinusCentredSet}, $(a+1)(f^*(r)+4r)=f(r)$ is a $2$-dimensional control function for $G$. Since $H^{\#}$ is isometric in $G$, $f$ is also a $2$-dimensional control function for $H^{\#}$, as desired.
    \end{proof}

    We can now bound the Assouad--Nagata dimension of $H$-minor-free classes of weighted graphs.

    \begin{theorem}
        \label{ANdimMCCWeighted}
        For any graph $H$ and any $H$-minor-free class of weighted graphs $\scr{G}$,
        \begin{enumerate}[label=(\alph*)]
            \item $\ANdim(\scr{G})\leq 2$; and
            \item $\ANdim(\scr{G})\leq 1$ if $\scr{G}$ has bounded treewidth.
        \end{enumerate}
    \end{theorem}

    \begin{proof}
        \cref{ANdimMCCWeighted}~(b) follows from \cref{ANdimTw}, so we focus on \cref{ANdimMCCWeighted}~(a).
    
        By \cref{GMST}, there exists an integer $k\geq 0$ such that every $H$-minor-free graph has a tree-decomposition of adhesion at most $k$ such that every torso is $(k,k,k,k)$-almost embeddable. Let $\scr{H}$ be the class of weighted graphs that are an induced subgraph of a $(k,k,k,k)$-almost embeddable weighted graph; by \cref{ANdimAlmostEmbeddable} $\scr{H}$ is hereditary and $\ANdim(\scr{H})\leq 2$. Observe that if every torso of a weighted graph $G$ is $(k,k,k,k)$-almost embeddable, then so is every weighted torso. Thus, every $G\in \scr{G}$ is $(k,\scr{H})$-constructable. By \cref{ANdimConstructable}, $\ANdim(\scr{G})\leq \max(\ANdim(\scr{H}),1)\leq 2$.
    \end{proof}

    \cref{ANdimMCC} then follows as a consequence of \cref{ANdimMCCWeighted} and \cref{asdimMCC}~(b), as the Assouad--Nagata dimension is at least the asymptotic dimension. We remark that the converse of \cref{ANdimMCCWeighted}~(b) is not true; for example, take the class $\scr{G}$ of $2$-dimensional grids (which is $K_5$ and $K_{3,3}$-minor-free as all $2$-dimensional grids are planar), and for each $G\in \scr{G}$, fix $r\in V(G)$, and weight each $uv\in E(G)$ by $2^{\dist_G(r,\{u,v\})}$. It can be seen that any $\ell$-path in the weighted graph is either contained in $2\ell$-neighbourhood of $r$, or consists of only a single vertex. It follows that $\ell\mapsto 4\ell$ is a $0$-dimensional control function for the corresponding class of weighted graphs $\scr{G}'$, despite $\scr{G}'$ also having unbounded treewidth \citep{RS5}.

    \section{Assouad--Nagata Dimension of Non-minor-closed Classes}
    We close this paper by looking at classes of weighted graphs that admit every graph as a minor. In particular, we establish some necessary conditions for such a class to have bounded asymptotic or Assouad--Nagata dimension, which leads to a proof of \cref{ANdimSubdivide}. We achieve this by migrating a control function for a graph $G$ in the class to a control function for any minor of $G$, achieved by reweighting edges appropriately.

    We require the following observation.
    \begin{observation}
        \label{remapping}
        Let $n\geq 0$ be an integer, let $G,H$ be weighted graphs, and let $f$ be an $n$-dimensional control function for $G$. If there exist real numbers $\alpha,\beta > 0$ and a map $\iota:V(H)\rightarrow V(G)$ such that for all $u,v\in V(H)$, $\beta\dist_H(u,v)\leq \dist_G(\iota(u),\iota(v)) \leq \alpha \dist_H(u,v)$, then $r\mapsto \frac{1}{\beta} f(\alpha r)$ is an $n$-dimensional control function for $H$.
    \end{observation}
    \begin{proof}
        Fix a real number $r>0$, let $c_G$ be an $(n+1,\alpha r,f(\alpha r))$-colouring of $G$. Define a colouring $c_H$ of $H$ via $c_H(v):=c_G(\iota(v))$ for each $v\in V(H)$. Since $\dist_G(\iota(u),\iota(v)) \leq \alpha \dist_H(u,v)$, for any monochromatic $r$-path $v_0,v_1,\dots,v_m$ in $H$ under $c_H$, $\iota(v_0),\iota(v_1),\dots,\iota(v_m)$ is a monochromatic $\alpha r$-walk in $G$ under $c_G$, and thus $\dist_G(\iota(v_0),\iota(v_n))\leq f(\alpha r)$. Since $\beta\dist_H(v_0,v_n)\leq \dist_G(\iota(v_0),\iota(v_n))$, we have that $\dist_H(v_0,v_n)\leq \frac{1}{\beta} f(\alpha r)$. It follows that $c_H$ is an $(n+1,r,\frac{1}{\beta} f(\alpha r))$-colouring of $H$.
    \end{proof}

    Notice that the fact that being isometric preserves control functions is a special case of \cref{remapping}, using $\alpha=\beta:=1$ and the identity map for $\iota$.

    We also make use of the following characterisation of minors. A \defn{model} of a graph $H$ in a graph $G$ is a partition $\scr{P}:=(P_x:x\in V(H))$ of some subgraph $A$ of $G$ such that $A/\scr{P}$ is isomorphic to $H$. Such a model exists if and only if $H\leq G$.

    We first show that for any graphs $G,H$ with $H$ connected and $H\leq G$, there is a $G'$ obtained from imbuing $G$ with rational edge weights whose control functions translate to control functions for $H$.

    \begin{proposition}
        \label{minorRealWeights}
        Let $\eps>0$ be a real number, let $G$ be a graph, and let $H$ be a connected unweighted graph that is a minor of $G$. Then there exists a function $w:E(G)\rightarrow \ds{R}^+$ such that the following holds. For any integer $n\geq 0$, if $f:\ds{R}^+\rightarrow \ds{R}^+$ is an $n$-dimensional control function for the weighted graph $G'$ obtained from imbuing $G$ with $w$, then the function $r\mapsto f((1+\eps)r)$ is an $n$-dimensional control function for $H$. Furthermore, if $\eps$ is rational, then $w$ is a rational-valued function.
    \end{proposition}
    \begin{proof}

        Let $\scr{P}:=(P_x:x\in V(H))$ be a model of $H$ in $G$, which exists as $H\leq G$. By definition, $\scr{P}$ is a partition of some subgraph $H'$ of $G$ such that $H'/\scr{P}$ is isomorphic to $H$. 
        
        Let $E$ be the set of edges $uv\in E(H')$ such that $u,v$ belong to different parts of $\scr{P}$, let $E'$ be the set of edges $uv\in E(H')$ such that $u,v$ belong to the same part of $\scr{P}$, and let $E'':=E(G)\setminus E(H')$. Define $\iota:V(H)\rightarrow V(G)$ via setting, for each $x\in V(H)$, $\iota(x)\in P_x$. Set $p:=\max(|E'|,1)$ and $d:=\diam(H)$; note that $d<\infty$ since $H$ is connected. Define $w$ via

        \[w(e):= \begin{cases} 
            1 & \text{if }v\in E, \\
            \frac{\eps}{p} & \text{if }v\in E', \\
            (1+\eps)d+1 & \text{otherwise.}
            \end{cases}
        \]

        Note that if $\eps$ is rational, then $w$ is rational-valued, as $d$ and $p$ are integers; the former being because $H$ is unweighted. We show that $w$ is the desired weighting via showing that $\dist_{H}(u,v)\leq \dist_{G'}(\iota(u),\iota(v)) \leq (1+\eps) \dist_H(u,v)$ for any $u,v\in V(H)$; the desired result then follows from \cref{remapping}. This is clearly true when $u=v$, so assume they are distinct, and set $\ell:=\dist_H(u,v)$, $\ell':=\dist_{G'}(\iota(u),\iota(v))$; note that $\ell\geq 1$ as $H$ is unweighted. Let $P$ be a shortest path from $u$ to $v$ in $H$, observe that $P$ corresponds to a path $P_G$ in $G'$ from $\iota(u)$ to $\iota(v)$ that uses exactly $\ell$ edges in $E$ and no edges in $E''$. Thus, the length of $P_G$ is at most $\ell+|E'|\frac{\eps}{p}\leq \ell+p\frac{\eps}{p}\leq (1+\eps)\ell$ as $\ell\geq 1$. Thus, $\ell'\leq (1+\eps)\ell$. Next, since $d=\diam(H)$, by the prior observation $\ell'\leq (1+\eps)d$. Therefore, if $P'$ is a shortest path from $\iota(u)$ to $\iota(v)$ in $G'$, then $P'$ does not use any edges in $E''$, as they all have weight larger than $\ell'$. Thus, $P'$ corresponds to a path in $H$ of length at most $\ell'$. This gives $\ell \leq \ell'$, as desired.
    \end{proof}

    We now show that there is a $G''$ obtained from imbuing $G$ with integer edge weights such that control functions for $G''$ translate to control functions for $G'$.

    \begin{proposition}
        \label{rationalToIntegerWeights}
        For every dilation $f$, there exists a dilation $f'$ such that the following holds. Let $G$ be a graph, and let $H$ be the result of imbuing $G$ with some rational-valued weighting function $w$. Then there exists a function $w':E(G)\rightarrow \ds{N}$ such that for any integer $n\geq 0$, if $f$ is an $n$-dimensional control function for the weighted graph $G'$ that is the result of imbuing $G$ with $w'$, then $f'$ is an $n$-dimensional control function for $H$.
    \end{proposition}

    \begin{proof}
        Since $f$ is a dilation, there exists a real number $\alpha>0$ such that $f(r)\leq \alpha r$ for every real number $r>0$; set $f'(r):=\alpha r$ for every real number $r>0$. We claim that $f'$ is the desired function.
    
        Since $w$ is rational-valued and $|E(G)|<\infty$, there exists some integer $k>0$ such that $kw$ is integer-valued; set $w':=kw$. Observe that for every $u,v\in V(H)=V(G')$, $k\dist_H(u,v)=\dist_{G'}(u,v)$. Thus, if $f$ is an $n$-dimensional control function for $G$, by \cref{remapping} the function $r\mapsto \frac{1}{k}f(kr)$ is an $n$-dimensional control function for $H$. Since $f(r)\leq \alpha r$ for every real number $r>0$, $\frac{1}{k}f(kr)\leq \frac{1}{k}\alpha (kr)=\alpha r=f'(r)$. Thus, $f'$ is an $n$-dimensional control function for $H$, as desired.
    \end{proof}

    The last step is to find a $G'''\in \scr{G}$ whose control functions translate to control functions for $G''$, this is achieved by subdividing $G$ according to the weights of $G''$. Before we put all this together formally, we introduce one more piece of terminology. Given a set $S\subseteq \ds{R}^{+}$, the \defn{$S$-weights completion} of a class of graphs $\scr{G}$, denoted \defn{$\scr{G}_S$}, is the class of weighted graphs $G$ such that all edge weights of $G$ are in $S$ and $\UNW(G)\in \scr{G}$.
    
    \begin{theorem}
        \label{dimNonMCC}
        For every class of graphs $\scr{G}$  that admits every graph as a minor,
        \begin{enumerate}[label=(\alph*)]
            \item $\asdim(\scr{G}_{\ds{R}^+})=\infty$;
            \item $\ANdim(\scr{G}_\ds{N})=\infty$; and
            \item if $\scr{G}$ is closed under subdivision, then $\ANdim(\scr{G})=\infty$.
        \end{enumerate}
    \end{theorem}

    \begin{proof}
        Let $\scr{H}$ be any class of connected unweighted graphs with infinite asymptotic (and Assouad--Nagata) dimension, such as the class of $(d,m)$-grids for all integers $d,m\geq 1$ \citep{Gromov1993}. Fix any integer $n\geq 0$ and any function $f:\ds{R}^+\rightarrow \ds{R}^+$. We show that if $f$ is an $n$-dimensional control function for $\scr{G}_{\ds{R}^+}$, then there is an $n$-dimensional control function for $\scr{H}$, a contradiction by choice of $\scr{H}$; \cref{dimNonMCC}~(a) follows. Additionally, we show that if $f$ is a dilation and an $n$-dimensional control function for $\scr{G}_\ds{N}$, we again have that there is an $n$-dimensional control function for $\scr{H}$, another contradiction; \cref{dimNonMCC}~(b) follows. Finally, we show that if $\scr{G}$ is closed under subdivision, and $f$ is a dilation and an $n$-dimensional control function for $\scr{G}$, then $f$ is a dilation that is an $n$-dimensional control function for $\scr{G}_\ds{N}$, contradicting \cref{dimNonMCC}~(b); \cref{dimNonMCC}~(c) follows.
        
        We start with the first claims. Presume that $f$ is an $n$-dimensional control function for $\scr{G}_{\ds{R}^+}$, and fix a rational number $\eps>0$. Now, for every $H\in \scr{H}$, there exists $G\in \scr{G}$ that contains $H$ as minor. By \cref{minorRealWeights}, there exists a rational-valued function $w:E(G)\rightarrow \ds{R}^+$ such that if $f$ is an $n$-dimensional control function for the weighted graph $G'$ obtained by imbuing $G$ with $w$, then the function $r\mapsto f((1+\eps)r)$ is an $n$-dimensional control function for $H$. Observe that $G'\in \scr{G}_{\ds{R}^+}$, thus $f$ is an $n$-dimensional function for $G'$ and $r\mapsto f((1+\eps)r)$ is an $n$-dimensional control function for $H$. Since this holds for every $H\in \scr{H}$, this would make $r\mapsto f((1+\eps)r)$ an $n$-dimensional control function for $\scr{H}$, as claimed. This completes the proof of \cref{dimNonMCC}~(a).

        Now, we tackle the second claim. Presume that $f$ is a dilation, and that $f$ is an $n$-dimensional control function for $\scr{G}_{\ds{N}}$. Since $f$ is a dilation, we can apply \cref{rationalToIntegerWeights}, let $f'$ be the function we obtain. For every $H\in \scr{H}$, let $G$ and $G'$ be defined as in the previous paragraph; note that if $f'$ is an $n$-dimensional control function for $G'$, then the function $r\mapsto f'((1+\eps)r)$ is an $n$-dimensional control function for $H$. So we aim to show that $f'$ is indeed an $n$-dimensional control function for $G'$. 
        
        Since $G'$ is the result of imbuing $G$ by the rational-valued weighting function $w$, by \cref{rationalToIntegerWeights} (using $G'$ as $H$), there exists a function $w':E(G)\rightarrow \ds{N}$ such that if $f$ is an $n$-dimensional control function for the weighted graph $G''$ obtained from imbuing $G$ with $w'$, then $f'$ is an $n$-dimensional control function for $G'$. Observe that $G''\in \scr{G}_{\ds{N}}$, thus $f$ is an $n$-dimensional control function for $G''$, $f'$ is an $n$-dimensional control function for $G'$, and $r\mapsto f'((1+\eps)r)$ is an $n$-dimensional control function for $H$. Since this holds for every $H\in \scr{H}$, this implies that $r\mapsto f'((1+\eps)r)$ would be an $n$-dimensional control function for $\scr{H}$, as claimed. This completes the proof of \cref{dimNonMCC}~(b).

        We now proceed to the final claim. Presume that $\scr{G}$ is closed under subdivision, that $f$ is a dilation, and that $f$ is an $n$-dimensional control function for $\scr{G}$. For every $G\in \scr{G}_{\ds{N}}$, let $w$ be the weighting of $G$. Then, let $G'$ be the graph obtained from $\UNW(G)$ by subdividing each $e\in E(\UNW(G))=E(G)$ $w(e)-1$ times. By definition of $\scr{G}_{\ds{N}}$, $\UNW(G)\in \scr{G}$, and because $\scr{G}$ is closed under subdivision, $G'$ is also in $\scr{G}$. Thus, $f$ is an $n$-dimensional control function for $G'$. Now, notice that $G$ is isometric in $G'$; this follows from the fact that for any $e\in E(G)$, $e$ has been split into a path of length exactly $w(e)$ in $G'$, and because no new paths have been created. Thus, $f$ is also an $n$-dimensional control function for $G$. Since this holds for every $G\in \scr{G}$, this implies that $f$ would be a $n$-dimensional control function for $\scr{G}$, as claimed. This completes the proof of \cref{dimNonMCC}~(c), and the overall proof of \cref{dimNonMCC}.
    \end{proof}

    \cref{ANdimSubdivide} follows directly from \cref{ANdimMCCWeighted} and \cref{dimNonMCC}~(c).
    
\paragraph{Acknowledgements.} The author thanks David Wood for his supervision and suggestions to improve this paper. The author acknowledges the work of Chun-Hung Liu~\citep{Liu23}, who independently proved \cref{ANdimMCC}.

\bibliography{Ref.bib}
    
\appendix
	
\section{Proof of \cref{dimToColouring}}

We recall \cref{dimToColouring} for convenience.

\dimCol*    
    
    This method of this proof is essentially identical to the proof of Proposition~1.17 in \citet*{Bonamy}, which shows implicitly that if $f(r)$ is an $n$-dimensional control function for $G$, then $G$ admits an $(n+1,r,rf(r))$-colouring for every real number $r>0$. Our improvement to an $(n+1,r,f(r))$-colouring comes not from any new idea, but rather due to only measuring distances in $G$ rather than both $G$ and $G^r$, and avoiding a conversion overestimate in the process. This improvement is significant as it gives an equivalent definition for control functions (and thus Assouad--Nagata dimension), as opposed to just an equivalent definition for asymptotic dimension.
    
    We now begin the proof of \cref{dimToColouring}.
	
    \begin{proof}
	    First, assume that $f(r)$ is an $n$-dimensional control function for $G$. Fix a real number $r>0$, by definition of $f$ there exist collections $\scr{C}_1,\dots,\scr{C}_{n+1}$ such that:

        \begin{enumerate}[label=(\alph*)]
            \item $\bigcup_{i=1}^{n+1}\bigcup_{S\in \scr{C}_i}S = V(G)$;
            \item $\scr{C}_i$ is $r$-disjoint for each $i\in \{1,\dots,n+1\}$; and
            \item $\wdiam_G(S)\leq f(r)$ for each $i\in \{1,\dots,n+1\}$ and $S\in C_i$.
        \end{enumerate}
     
        Let $c:V(G)\rightarrow \{1,\dots,n+1\}$ be defined such that $v\in \bigcup_{S\in \scr{C}_{c(v)}}S$; note that this is possible by property~(a), but not necessarily unique. Observe that for any $i\in \{1,\dots,n+1\}$, if $u,v\in V(G)$ are both coloured $i$ under $c$ and are at distance at most $r$ in $G$, then $u,v$ both belong to the same set of $\scr{C}_i$ by property~(b). It follows that any $i$-monochromatic $r$-path in $G$ under $c$ is contained in some set of $\scr{C}_i$, and thus has weak diameter in $G$ at most $f(r)$, by property~(c). Therefore, $c$ is an $(n+1,r,f(r))$-colouring, as desired.
        
	    For the reverse direction, fix a real number $r>0$, and let $c$ be a given $(n+1,r,f(r))$-colouring of $G$, which we may assume is with colours $\{1,\dots,n+1\}$. For each $i\in \{1,\dots,n+1\}$, let $\scr{C}_i$ be the set $\{V(M):M\text{ an } i\text{-monochromatic } r\text{-component under }c\}$. Observe that for any colour $i$, any two distinct $i$-monochromatic $r$-components $M_1,M_2$ of $G$ under $c$ must be $r$-disjoint. Otherwise, $M_1\cup M_2$ would be connected in $G^r$; since $V(M_1\cup M_2)$ is $i$-monochromatic and $M_1,M_2$ are strict subgraphs of $M_1\cup M_2$, this contradicts the maximality of $M_1$ and $M_2$. Thus, $\scr{C}_i$ is $r$-disjoint. By definition of $c$, for any monochromatic $r$-component $M$ of $G$ under $c$, $V(M)$ has weak diameter at most $f(r)$ in $G$. Additionally, each $v\in V(G)$ must be in some monochromatic $r$-component of $G$ (of the same colour as $v$), so $\bigcup_{i=1}^{n+1}\bigcup_{V(M)\in \scr{C}_i}V(M) = V(G)$. Therefore, $\scr{C}_1,\dots,\scr{C}_{n+1}$ are the desired collections. The result follows.
    \end{proof}
    
    \section{Proof of \cref{ANdimInfinite}}

    We recall \cref{ANdimInfinite} for ease of reference.

    \ANdimInf*

    This theorem is very similar to that Theorem~A$.2$ of \citet*{Bonamy}. The difference is that we obtain the control function $r\mapsto f((1+\eps)r)$ for $G$, whereas Theorem~A$.2$ of \citet*{Bonamy} gives the control function $r\mapsto f(r+1)$. This difference is significant, as if $f$ is a dilation, then so is $r\mapsto f((1+\eps)r)$; however $r\mapsto f(r+1)$ need not be a dilation. This allows us to give an upper bound on the Assouad--Nagata dimension instead of just the asymptotic dimension.

    Before proceeding to the proof of \cref{ANdimInfinite}, we observe a key difference between infinite and finite graphs. In a finite graph, a shortest path between any two vertices always exists, since there are only finitely many paths. However, this is not true in an infinite graph. Consequently, in an infinite graph, the distance between two vertices is instead the infimum across all lengths of paths between the vertices. However, we remark that most other ideas translate to the infinite case without issue; importantly, objects such as $r$-paths and monochromatic $r$-components can be defined identically. Additionally, while not explicitly stated, \cref{dimToColouring} holds even when $G$ is infinite, using an identical proof.
    
    The proof for \cref{ANdimInfinite} is essentially identical to the proof of Theorem~A$.2$ of \citet*{Bonamy}. The improved control function for $G$ comes from a slight tweak in the error of the length of a path that is very close to the infimum.
    
    We require a technical lemma that is a special case of \citet*{Gottschalk1951}.

    \begin{lemma}
        \label{globalChoice}
        Let $C,I$ be sets, let $\scr{S}$ be the set of all finite subsets of $I$, and for each $S\in \scr{S}$, let $c_S:S\rightarrow C$. Then there exists a map $c:I\rightarrow C$ such that for every $S\in \scr{S}$, there exists $S'\in \scr{S}$ with $S\subseteq S'$ such that $c_{S'}=c\big|_{S'}$.
    \end{lemma}

    We now begin the proof of \cref{ANdimInfinite}.

    \begin{proof}
        For any real number $r>0$, let $f'(r):=f((1+\eps)r)$; we must find an $(n+1,r,f'(r))$ of $G$. Set $C:=\{1,\dots,n+1\}$, and let $\scr{S}$ be the set of all finite subsets of $I:=V(G)$. For each $S\in \scr{S}$, let $c_S$ be an $(n+1,(1+\eps)r,f'(r))$-colouring of $G[S]$, which we may assume is with colours $C$. We can now apply \cref{globalChoice} to obtain a map $c:V(G)\rightarrow C$ such that for each $S\in \scr{S}$, there exists $S'\in \scr{S}$ with $c_{S'}=c\big|_{S'}$. We show that $c$ is an $(n+1,r,f'(r))$-colouring of $G$.

        Let $P$ be any monochromatic $r$-path in $G$ under $c$. Since consecutive vertices in $P$ are at distance at most $r$ in $G$, we can find a path of length at most $(1+\eps)r$ between them in $G$. Stringing these paths together gives us a walk $P'$ containing $P$ such that consecutive vertices of $P$ are at distance at most $(1+\eps)r$ in $G[P']$. Let $S'\in \scr{S}$ be such that $P'\subseteq S'$ and $c_{S'}=c\big|_{S'}$. Since $G[P']\subseteq G[S']$, consecutive vertices of $P$ are also at distance at most $(1+\eps)r$ in $G[S']$, and since $P$ is monochromatic in $G$ under $c$ and $c_{S'}=c\big|_{S'}$, $P$ is also monochromatic in $G[S']$ under $c_{S'}$. Hence, $P$ forms a monochromatic $((1+\eps)r)$-path in $G[S']$ under $c_{S'}$, and thus it has weak diameter in $G[S']$ and $G$ is at most $f'(r)$, as desired.
    \end{proof}

    It follows from \cref{ANdimInfinite} that the Assouad--Nagata dimension of an infinite weighted graph $G$ is at most the Assouad--Nagata dimension of the class of all finite induced subgraphs of $G$. We remark that the converse is not true; for example, take the disjoint union of all graphs in any class of finite unweighted graphs with infinite Assouad--Nagata dimension, such as the class of $(d,m)$-grids for all integers $d,m\geq 1$ \citep{Gromov1993}, and then add a single vertex adjacent to every other vertex. This infinite graph $G$ has Assouad--Nagata dimension 0 as $\wdiam_G(V(G))\leq 2$, but the class of induced subgraphs of $G$ contains a class with infinite Assouad--Nagata dimension, and thus also has infinite Assouad--Nagata dimension. We also remark that an analogous argument can be used to show that the converse is not true even when Assouad--Nagata dimension is replaced by asymptotic dimension.
    
\end{document}